\documentclass{amsart}
\usepackage[all,2cell]{xy}
\UseAllTwocells
\usepackage{amssymb,amsmath,stmaryrd,mathrsfs,amsthm}
\usepackage{color}
\definecolor{darkgreen}{rgb}{0,0.45,0} 
\usepackage[pagebackref,colorlinks,citecolor=darkgreen,linkcolor=darkgreen]{hyperref}
\usepackage{url}
\usepackage{graphicx}
\usepackage{mathtools}

\makeatletter
\let\ea\expandafter

\newcount\foreachcount
\def\foreachletter#1#2#3{\foreachcount=#1
  \ea\loop\ea\ea\ea#3\@alph\foreachcount
  \advance\foreachcount by 1
  \ifnum\foreachcount<#2\repeat}
\def\foreachLetter#1#2#3{\foreachcount=#1
  \ea\loop\ea\ea\ea#3\@Alph\foreachcount
  \advance\foreachcount by 1
  \ifnum\foreachcount<#2\repeat}
\def\definescr#1{\ea\gdef\csname s#1\endcsname{\ensuremath{\mathscr{#1}}}}
\foreachLetter{1}{27}{\definescr}
\def\definecal#1{\ea\gdef\csname c#1\endcsname{\ensuremath{\mathcal{#1}}}}
\foreachLetter{1}{27}{\definecal}
\def\definebold#1{\ea\gdef\csname b#1\endcsname{\ensuremath{\mathbf{#1}}}}
\foreachLetter{1}{27}{\definebold}
\def\defineul#1{\ea\gdef\csname u#1\endcsname{\ensuremath{\underline{\smash{\mathsf{#1}}}}}}
\foreachLetter{1}{27}{\defineul}
\foreachletter{1}{27}{\defineul}

\let\al\alpha

\newcommand{\ten}{\ensuremath{\otimes}}
\newcommand{\inv}{^{-1}}
\newcommand{\sm}{\mathbin{\wedge}}
\newcommand{\exsm}{\mathbin{\overline{\wedge}}}

\newcommand{\ep}{\ensuremath{\varepsilon}}
\newcommand{\ph}{\ensuremath{\varphi}}

\newcommand{\Id}{\ensuremath{\operatorname{Id}}}
\newcommand{\id}{\ensuremath{\operatorname{id}}}
\newcommand{\Ho}{\ensuremath{\operatorname{Ho}}}
\newcommand{\Hom}{\ensuremath{\operatorname{Hom}}}

\newcommand{\op}{\ensuremath{^{\mathit{op}}}}

\newcommand{\tr}{\ensuremath{^{\top}}}

\newcommand{\adj}{\dashv}
\newcommand{\iso}{\cong}
\newcommand{\too}[1][]{\ensuremath{\overset{#1}{\longrightarrow}}}
\newcommand{\toot}{\ensuremath{\rightleftarrows}}
\newcommand{\toto}{\ensuremath{\rightrightarrows}}

\newcommand{\leftwe}{\ensuremath{\overset{\sim}{\longleftarrow}}}
\newcommand{\rightwe}{\ensuremath{\overset{\sim}{\longrightarrow}}}
\newcommand{\maps}{\colon}
\newcommand{\spam}{\,:\!}

\newcommand{\pullbackcorner}[1][dr]{\save*!/#1-1.2pc/#1:(-1,1)@^{|-}\restore}
\let\xto\xrightarrow

\newif\ifhyperref
\@ifpackageloaded{hyperref}{\hyperreftrue}{\hyperreffalse}
\ifhyperref
  \def\defthm#1#2{%
    \newtheorem{#1}{#2}[section]%
    \expandafter\def\csname #1autorefname\endcsname{#2}%
    \expandafter\let\csname c@#1\endcsname\c@thm}
\else
  \def\defthm#1#2{\newtheorem{#1}[thm]{#2}}
  \ifx\SK@label\undefined\let\SK@label\label\fi
  \let\old@label\label
  \let\your@thm\@thm
  \def\@thm#1#2#3{\gdef\currthmtype{#3}\your@thm{#1}{#2}{#3}}
  \def\currthmtype{}
  \def\label#1{{\let\your@currentlabel\@currentlabel\def\@currentlabel%
      {\currthmtype~\your@currentlabel}%
      \SK@label{#1@}}\old@label{#1}}
  \def\autoref#1{\ref{#1@}}
\fi

\newtheorem{thm}{Theorem}[section]

\defthm{cor}{Corollary}
\defthm{prop}{Proposition}
\defthm{lem}{Lemma}
\theoremstyle{definition}
\defthm{defn}{Definition}
\defthm{notn}{Notation}
\theoremstyle{remark}
\defthm{rmk}{Remark}
\defthm{eg}{Example}
\defthm{egs}{Examples}

\newenvironment{blist}{\begin{list}{\labelitemi}{\leftmargin=1.5em\labelwidth=1em}}{\end{list}}

\let\c@equation\c@thm
\numberwithin{equation}{section}

\def\autofmt@b#1\autofmt@end{\mathbf{#1}}
\def\autofmt@c#1#2\autofmt@end{\mathcal{#1}\mathit{#2}}
\def\autofmt@s#1#2\autofmt@end{\mathscr{#1}\mathit{#2}}
\def\autofmt@f#1\autofmt@end{\mathfrak{#1}}
\def\autofmt@u#1\autofmt@end{\underline{\smash{\mathsf{#1}}}}

\def\auto@drop#1{}
\def\autodef#1{\ea\ea\ea\@autodef\ea\ea\ea#1\ea\auto@drop\string#1\autodef@end}
\def\@autodef#1#2#3\autodef@end{%
  \ea\def\ea#1\ea{\ea\ensuremath\ea{\csname autofmt@#2\endcsname#3\autofmt@end}}}
\def\autodefs@end{blarg!}
\def\autodefs#1{\@autodefs#1\autodefs@end}
\def\@autodefs#1{\ifx#1\autodefs@end%
  \def\autodefs@next{}%
  \else%
  \def\autodefs@next{\autodef#1\@autodefs}%
  \fi\autodefs@next}

\autodefs{\cCat\cStr\cConj\fDbl\uModel\uCat\bEx\bTop\bSp\uCat\uModel\uDrv
  \bCh\cAlg\uAlg\fCat\cModel\uSq\uMonCat\cMonCat\bSh\uCrs\uQCat\cQCat
}

\newcommand{\ftwoCat}{\ensuremath{\mathsf{2}\text{-}\fCat}}

\newcommand{\cModell}{\ensuremath{\cModel_L}}
\newcommand{\cModelr}{\ensuremath{\cModel_R}}

\newcommand{\bbTAlg}{\ensuremath{T\text-\uAlg}}
\newcommand{\cTAlg}{\ensuremath{T\text-\cAlg}}

\newcommand{\conj}{\Yleft}
\newcommand{\comp}{\approxeq}

\makeatother

\title{Comparing composites of left and right derived functors}
\author{Michael Shulman}
\email{mshulman@ucsd.edu}
\begin{document}
\maketitle

\begin{abstract}
  We introduce a new categorical framework for studying derived
  functors, and in particular for comparing composites of left and
  right derived functors.  Our central observation is that model
  categories are the objects of a double category whose vertical and
  horizontal arrows are left and right Quillen functors, respectively,
  and that passage to derived functors is functorial at the level of
  this double category.  The theory of conjunctions and mates in
  double categories, which generalizes the theory of adjunctions and
  mates in 2-categories, then gives us canonical ways to compare
  composites of left and right derived functors.  We give a number of
  sample applications, most of which are improvements of existing
  proofs in the literature.
\end{abstract}

\tableofcontents

{\renewcommand\thefootnote{}
\footnotetext{This paper was published in the \emph{New York Journal of Mathematics}, Vol.~17 (2011), p75--125.  The published version is available at \url{http://nyjm.albany.edu/j/2011/17-5.html}.}}

\part{Theory}
\label{sec:theory}

\section{Introduction}
\label{sec:introduction}

Part of the general philosophy of category theory is that morphisms
are often more important and subtler than objects.  This applies also
to categories and functors themselves, as well as to more complicated
categorical structures, such as model categories and Quillen
adjunctions.  The passage from a model category to its homotopy
category is well understood, but the passage from Quillen functors to
derived functors seems more subtle and mysterious.  In particular, the
distinction between \emph{left} and \emph{right} derived functors is
not well understood at a conceptual level.

For instance, it is well-known that taking derived functors of Quillen
functors between model categories is pseudofunctorial---as long as all
derived functors involved have the same ``handedness.''  In other words, we
have coherent isomorphisms such as $\bL(G\circ F) \iso \bL G \circ \bL F$.
However, not infrequently it happens that we want to compose a Quillen
left adjoint with a Quillen right adjoint, and compare the result with
another such composite.  Standard model category theory has little to
say about such questions, but such comparisons are often essential in
applications.

For example, the authors of~\cite{maysig:pht} construct, for every
base space $B$, a model category $\bEx_B$ of ex-spaces over $B$, and
for every continuous map $f\maps A\to B$, a Quillen adjunction
$f_!\maps \bEx_A \toot \bEx_B \spam f^*$.  The derived functor $\bL
f_!$ is a parametrized version of homology, while the additional right
adjoint $f_*$ of $\bR f^*$ (which is shown to exist using Brown
representability) is a parametrized version of cohomology.  One of
their central lemmas about these adjunctions is that for any pullback
square
\[\vcenter{\xymatrix{A \ar[r]^h\ar[d]_f \pullbackcorner & B \ar[d]^g\\
    C\ar[r]_k & D}}
\]
of base spaces, there are isomorphisms $\bL f_!\circ \bR h^* \iso \bR
k^*\circ \bL g_!$ of derived functors, as long as either $g$ or $k$ is
a fibration.  (This sort of result is sometimes called a
\emph{Beck-Chevalley condition}.)  Another analogous result is the
proper base change theorem in sheaf theory.  The aim of the current
paper is to provide a general categorical framework in which to speak
about such comparisons.

We should stress at the outset that we will not prove any general theorem
about when two composites of left and right derived
functors are isomorphic.  Like the question of whether a given Quillen
adjunction is a Quillen equivalence, the way to attack this question
seems to depend a great deal on the particular situation.  What we do
give is a calculus describing the relationships between the natural
transformations which compare such composites, which generalizes the
calculus of ``mates'' in 2-categories (see \S\ref{sec:mates}).
This gives a general framework in which to speak about such
comparisons, and a way to make them more precise by identifying the
particular natural transformation which is an isomorphism.

Many well-known theorems in homotopy theory and homological algebra,
such as the proper base change theorem and the lemma
from~\cite{maysig:pht} mentioned above, can be restated in this
language.  We will work out a number of such examples in the second
half of the paper.  Usually, the existing proofs of such theorems
already include all the real work, and only a little reformulation is
required to place them in our abstract context.  There are two main
advantages accruing from the small amount of work involved in such a
reformulation.
\begin{enumerate}
\item When trying to prove a new commutation result between left and
  right derived functors, the abstract framework reduces the problem
  from \emph{finding} an isomorphism to \emph{proving} that a
  particular, canonically specified map is a weak equivalence.
\item The result obtained is strengthened from the mere
  \emph{existence} of an isomorphism, which is usually all that the
  standard proofs provide, to the statement that a specific
  \emph{canonically defined} map is an isomorphism.  In some cases
  this does not matter, but in others (for instance, when coherence
  questions arise) it does.
\end{enumerate}
In addition, we believe that our abstract framework sheds conceptual
light on the distinction between left and right derived functors.

The central idea of this paper is to upgrade the category (or
2-category) of model categories and Quillen adjunctions to a more
expressive structure called a \emph{double category}.
The standard 2-category of model categories
and Quillen adjunctions is a somewhat uncomfortable thing,
since to define it one must choose whether to
consider a Quillen adjunction as pointing in the direction of the
right adjoint or the left adjoint, and either choice is asymmetrical
and aesthetically unsatisfactory.  A double category, on the other
hand, can include both the left and right Quillen functors as
different types of morphism.  Quillen adjunctions then appear as
``conjunctions'' in this double category.  The central observation
enabling us to compare left and right derived functors is that the
passage from Quillen functors to derived functors is a functor of
double categories.

Category theorists will be interested to see that there is also a
formal analogy between \emph{left} Quillen functors and \emph{colax}
monoidal functors (or colax morphisms for any 2-monad), and between
\emph{right} Quillen functors and \emph{lax} monoidal functors.  A
functor which is both left and right Quillen corresponds to a strong
monoidal functor, while a Quillen adjunction corresponds to a
``doctrinal'' or ``lax/colax adjunction.''

The plan of this paper is as follows.  In the first part, comprising
\S\S\ref{sec:mates}--\ref{sec:derivable-categories}, we develop the
general theory.  We begin in \S\ref{sec:mates} by
summarizing the theory of \emph{mates} in 2-categories, which provides a
general way to construct transformations comparing composites of
adjoints; this is familiar to category theorists, but not as widely
known as it should be.  Then in \S\ref{sec:der-func} we summarize the
standard theory of derived functors and note its deficiencies.  The
next three sections are devoted to setting up the double-categorical
machinery we need.  In \S\ref{sec:dblcats} we recall the notion of
double category and give our main examples, including the double
category of model categories.  In \S\ref{sec:compconj} we describe the
theory of \emph{companions} and \emph{conjoints} in double categories,
which generalizes the calculus of adjunctions and mates in
2-categories.  And in \S\ref{sec:double-psfrs} we define the relevant type
of functor between double categories, which we call a \emph{double
  pseudofunctor}; it differs from the most common notions of functor
between double categories in that it only preserves the structure weakly in
\emph{both} directions.

The main result of the paper is proven in
\S\ref{sec:derived-functors}: passage to homotopy categories and
derived functors is a double pseudofunctor on the double category of
model categories.  In \S\ref{sec:derivable-categories} we extend
this result to a more general context, involving categories equipped
with ``cofibrant'' and ``fibrant'' approximations that admit more
flexibility than those in a model structure; this generality turns out
to be important in many examples.  We call these \emph{derivable
  categories} and the morphisms between them \emph{derivable
  functors}; they are closely related to the \emph{deformable
  functors} of~\cite{dhks:holim}.

In the second part of the paper, comprising
\S\S\ref{sec:beck-chev-cond}--\ref{sec:projfmla-ex}, we work out a
number of example applications.  Our goal is to show how the general theory can be applied in
practice to compare composites of left and right derived functors, and
to provide templates for future applications.  In most of the
examples we consider, the \emph{existence} of an isomorphism is known;
our contribution is to put all these facts in a general framework and
show that the isomorphisms involved are actually the canonically
defined maps which one would hope to be isomorphisms.

The procedure in all these examples is the following: use the general
theory to identify a point-set-level representative of the canonically
defined map in question, then invoke facts specific to the domain at
hand to show that this map is (or, in some cases, is not) a weak
equivalence.  In general, the application of the general theory is
easy, and the domain-specific facts are the same ones used in the
standard proofs that an isomorphism exists (so that, in particular,
the isomorphism constructed in the classical proof is in fact the
canonical one).  We do, however, include one example in
\S\ref{sec:projfmla-ex} where the general theory does not apply so
cleanly and a medium-sized diagram chase is still required.  But even
in this case, the general theory simplifies the problem significantly
and provides a context in which to ask the right questions.

Our reference for model category theory is~\cite{hovey:modelcats}; in
particular, we assume our model categories to be equipped with
functorial factorizations.  This is not strictly necessary, but it
will make things easier.  A good reference for the 2-category theory
we will need is the first few sections of~\cite{ks:r2cats}.

I would like to thank my thesis advisor, Peter May, for helpful
conversations about derived functors, and the referee, for pointing
out that more concrete examples were necessary.

\section{Mates}
\label{sec:mates}

Since left and right derived functors are, in particular, left and
right adjoints, we begin by considering how to compare composites of
left and right adjoints.  The primary tool used for this purpose is
the theory of \emph{mates} in 2-categories.  Though straightforward,
this theory is not as well-known as it should be, and is thus
frequently reinvented.  Here we give a brief overview; a definitive
treatment can be found in~\cite{ks:r2cats}.

The most basic form of the mate correspondence says that if
$f^*,g^*\maps \sA\toto \sB$ are parallel functors with left adjoints $f_!$
and $g_!$, respectively, then there is a bijection between natural
transformations $f_!\to g_!$ and natural transformations $g^*\to f^*$.
A pair of natural transformations that correspond to each other under
this bijection are called \textbf{mates} (or sometimes ``conjugates'' or
``adjuncts'').

More generally, for functors $f^*\maps \sA\to \sB$ and $g^*\maps
\sC\to \sD$ with left adjoints $f_!$ and $g_!$, and any functors
$h^*\maps \sD\to \sB$ and $k^*\maps \sC\to \sA$, there is a bijection
between natural transformations $f_! h^* \to k^* g_!$ and natural
transformations $h^* g^* \to f^* k^*$, i.e.\ between transformations
\[\vcenter{\xymatrix{\sD \ar[r]^{h^*}\ar[d]_{g_!} \drtwocell\omit &
    \sB \ar[d]^{f_!}\\
    \sC\ar[r]_{k^*} & \sA}} \qquad\text{and}\qquad
\vcenter{\xymatrix{\sD \ar[r]^{h^*}\ar@{<-}[d]_{g^*}  &
    \sB \ar@{<-}[d]^{f^*}\\
    \sC\ar[r]_{k^*} \urtwocell\omit & \sA.}}
\]
Explicitly, the mate of $\alpha\maps f_! h^* \to k^* g_!$ is the composite
\[ h^* g^* \xto{\eta h^* g^*} f^* f_! h^* g^* \xto{f^* \alpha g^*} f^*
k^* g_! g^* \xto{f^*k^* \ep} f^*h^*,
\]
where $\eta$ is the unit of the adjunction $f_!\adj f^*$ and $\ep$ is
the counit of the adjunction $g_! \adj g^*$.  This is also commonly
described as the ``pasting composite''
\[\vcenter{\xymatrix{
    &
    \sD \ar@{}[]="D" \ar[r]^{h^*}\ar[d]^{g_!} \drtwocell\omit
    &
    \sB \ar[d]_{f_!} \ar@{=}[r]|{}="mid1" &
    \sB\\
    \sC \ar@{=}[r]|{}="mid2" \ar[ur]^{g^*}
    \ar@{}"D";"mid2"|(.6){\Downarrow}
    &
    \sC\ar[r]_{k^*} & \sA \ar[ur]_{f^*} \ar@{}[]="A"
    \ar@{}"mid1";"A" |(.4){\Downarrow}
  }}\]
Dually, the mate of $\beta\maps h^* g^* \to f^* k^*$ is the composite
\[ f_! h^* \xto{f_! h^* \eta} f_!h^*g^*g_! \xto{f_!\beta g_!}
f_!f^*k^*g_! \xto{\ep k^*g_!} k^* g_!.
\]
The triangular diagrams for the adjunctions $f_!\adj f^*$ and $g_!\adj
g^*$ are precisely what is required to make these into inverse
bijections.

In general, the mate of an isomorphism need not be an isomorphism, but
there are two important situations in which it is.

\begin{lem}\label{thm:hkid-mate-iso}
  If $h^*$ and $k^*$ are identities, then a transformation $f_!\to
  g_!$ is an isomorphism if and only if its mate $g^*\to f^*$ is an
  isomorphism.
\end{lem}
\begin{proof}
  The mate of the inverse of one supplies an inverse to the other.
\end{proof}

Note that this includes the standard fact that any two right adjoints
of a given functor are canonically isomorphic.

\begin{lem}\label{thm:adjeqv-mate-iso}
  If $f_!\adj f^*$ and $g_!\adj g^*$ are adjoint equivalences, then a
  transformation $f_!h^* \to k^*g_!$ is an isomorphism if and only if
  its mate $h^*g^* \to f^*k^*$ is an isomorphism.
\end{lem}
\begin{proof}
  In this case the $\eta$ and $\ep$ appearing in the definition of
  mates are isomorphisms, so composing with them preserves
  invertibility.
\end{proof}

However, in cases other than these, whether or not a given mate is an
isomorphism can have substantial mathematical content.  Here are two
examples; we will see more in
\S\S\ref{sec:beck-chev-cond}--\ref{sec:projfmla-ex}.

\begin{eg}\label{eg:beck-chev}
  In the situation of~\cite{maysig:pht} mentioned in the introduction,
  we have a category $\bEx_B$ associated to every space $B$ and an
  adjunction $f_!\maps \bEx_A\toot\bEx_B\spam f^*$ to every continuous
  map $f\maps A\to B$, and moreover for any commutative square
  \begin{equation}
    \vcenter{\xymatrix{A \ar[r]^h\ar[d]_f & B \ar[d]^g\\
        C\ar[r]_k & D}}\label{eq:bc-down}
  \end{equation}
  of continuous maps, we have a natural isomorphism
  \begin{equation}
    h^*g^*\xto{\iso} f^*k^*.\label{eq:bc-orig}
  \end{equation}
  This isomorphism has a mate
  \begin{equation}
    f_! h^* \too k^* g_!\label{eq:bc-mate}
  \end{equation}
  which may or may not be an isomorphism, depending on the
  square~\eqref{eq:bc-down}; in particular, it is an isomorphism
  whenever~\eqref{eq:bc-down} is a pullback square.  In category
  theory, the property of~\eqref{eq:bc-mate} being an isomorphism is
  often called the \emph{Beck-Chevalley condition} for the
  square~\eqref{eq:bc-down}.

  This example is paradigmatic of a very general situation: we have a
  category $\sS$ (here $\sS=\mathbf{Top}$) and a pseudofunctor
  $\sS\op\to\cCat$ (here this pseudofunctor sends $B$ to $\bEx_B$),
  with the property that each morphism $f$ of \sS\ is sent to a
  functor $f^*$ having a left adjoint $f_!$.  For any such
  pseudofunctor, we can ask whether a given commutative square in \sS\
  satisfies the Beck-Chevalley condition; often this is the case for
  (some class of) pullback squares in \sS.
\end{eg}

\begin{eg}\label{eg:closed-monfunc}
  Let \sC\ and \sD\ be closed symmetric monoidal categories and let
  $f^*\maps \sD\to\sC$ be a lax monoidal functor; this means that we
  have natural transformations
  \begin{align}
    f^*X\ten f^*Y &\too f^*(X\ten Y)\label{eq:lax-transf}\\
    I &\too f^*I
  \end{align}
  satisfying certain axioms.  Now we can also
  regard~\eqref{eq:lax-transf} as a transformation
  \[\vcenter{\xymatrix{\sD \ar[r]^{f^*}\ar[d]_{X\ten -} \drtwocell\omit & \sC \ar[d]^{f^*X\ten -}\\
      \sD\ar[r]_{f^*} & \sC}}
  \]
  which therefore has a mate
  \begin{equation}
    \label{eq:closed-transf}
    f^*\Hom(X,Y) \too \Hom(f^*X,f^*Y).
  \end{equation}
  We say that $f^*$ is a \emph{closed monoidal functor}
  if~\eqref{eq:closed-transf} is an isomorphism.

  Now suppose that $f^*$ has a left adjoint $f_!$.  Then we also have
  composite adjunctions
  \begin{align*}
    X\ten f_!(-)&\;\adj\; f^*\Hom(X,-) \qquad\text{and}\\
    f_!(f^*X \ten -)&\;\adj\; \Hom(f^*X, f^*-).
  \end{align*}
  Under these adjunctions,~\eqref{eq:closed-transf} has a mate
  \begin{equation}
    f_!(f^*X \ten A) \too X\ten f_!A.\label{eq:closed-transf-mate}
  \end{equation}
  By \autoref{thm:hkid-mate-iso},~\eqref{eq:closed-transf-mate} is an
  isomorphism if and only if~\eqref{eq:closed-transf} is.  This
  alternate condition is sometimes easier to verify.

  There are many other mates of this sort that compare various
  composites of adjoint functors between monoidal categories; see, for
  instance,~\cite{fhm:left-and-right}.
\end{eg}

The thing to notice about both of these examples is that the given
structure uniquely specifies a canonical transformation, and the
important question is whether that transformation is an isomorphism.
Thus, for instance, in the case of the Beck-Chevalley condition, it is
important not merely that there \emph{exists} an isomorphism $f_! h^*
\iso k^* g_!$, but that \emph{the particular transformation} $f_! h^*
\to k^* g_!$ (the mate of the specified isomorphism $h^*g^*\iso
f^*k^*$) is an isomorphism.  The mere existence of an isomorphism may
be sufficient for some applications, such as computing
homology and cohomology groups up to isomorphism.  However, for other purposes, such as
proving the coherence axioms for the bicategory of parametrized
spectra constructed in~\cite{maysig:pht} (see~\cite{shulman:frbi} for
details), it is essential to know \emph{what} that isomorphism is.

We end this section with a useful observation about iterated mates.

\begin{lem}\label{lem:adj-composite-mate}
  Given a transformation
  \[\vcenter{\xymatrix@-.5pc{\sA \ar[r]^{g^*}\ar[d]_{h^*} \drtwocell\omit & \sB \ar[d]^{k^*}\\
      \sC\ar[r]_{f^*} &\sD }}
  \]
  where all the functors $f^*$, $g^*$, $h^*$, and $k^*$ have left
  adjoints $f_!$, $g_!$, $h_!$, and $k_!$ respectively, if we first
  take its mate under the adjunctions $f_!\adj f^*$ and $g_!\adj g^*$
  to obtain a transformation $f_!k^* \to h^*g_!$, and then take the
  mate of this under the adjunctions $h_!\adj h^*$ and $k_!\adj k^*$,
  the resulting transformation $h_!f_!\to g_!k_!$ is the same as the
  mate of the original transformation under the composite adjunctions
  $h_!f_!\adj f^*h^*$ and $g_!k_!\adj k^*g^*$.
\end{lem}
\begin{proof}
  By unraveling definitions.
\end{proof}

\section{Derived functors}
\label{sec:der-func}

A good deal of the power of model category theory, and of abstract
homotopy theory more generally, comes from its ability to construct
\emph{derived} structure (that is, structure at the level of homotopy
categories) from \emph{point-set level} structure, in a tractable way.
The most basic example, of course, is the construction of homotopy
categories themselves; a few other examples include the constructions
of
\begin{enumerate}
\item derived functors from point-set level functors,
\item monoidal homotopy categories from monoidal model categories,
\item enriched homotopy categories from enriched model categories, and
\item triangulated homotopy categories from stable model categories.
\end{enumerate}
Of course, structure-preserving passage from one world to another is a
common phenomenon in mathematics; to describe it formally the term
\emph{functor} was invented.  In the case of constructing derived
structure, one general functoriality statement was proven
in~\cite{hovey:modelcats}: passage to derived functors is a
\emph{pseudofunctor} from a 2-category \cModel\ of model categories to
the 2-category \cCat\ of categories, functors, and natural
transformations.

In order to make such a statement precise, we need to specify what the
morphisms are in \cModel.  However, there are really two different
types of morphism between model categories, so we end up with two
different 2-categories.  In \cModell, the morphisms are left
Quillen functors (functors which preserve cofibrations and acyclic
cofibrations and have a right adjoint), and in \cModelr, the morphisms
are right Quillen functors.  In each case, we allow arbitrary natural
transformations as the 2-cells.

For the reader's convenience, we now recall the usual definition of
derived functors.  If $f\maps \sC\to\sD$ is left Quillen, by Ken
Brown's lemma it preserves weak equivalences between cofibrant
objects, so the composite $f\circ Q\maps \sC\to\sD$ preserves all weak
equivalences (where $Q$ denotes a functorial cofibrant replacement).
Thus, $f\circ Q$ induces a functor $\bL f\maps \Ho(\sC)\to\Ho(\sD)$
which we call the \textbf{left derived functor} of $f$.  Dually, a
right Quillen functor has a \textbf{right derived functor} $\bR f$
induced by $f\circ R$ (where $R$ denotes a functorial fibrant
replacement).

\begin{rmk}\label{rmk:der-kan}
  One can show that such a left derived functor of $f$ is, in
  particular, a right Kan extension of $f$ along the localization
  $\sC\to\Ho(\sC)$, and many authors take this as a definition of
  ``derived functor''.  From our point of view it is fairly
  irrelevant, although it does imply that $\bL f$ depends only on the
  weak equivalences in \sC\ rather than the model structure.
\end{rmk}

Recall that a \emph{pseudofunctor} between 2-categories (also called a
\emph{weak 2-functor} or a \emph{homomorphism of bicategories}) is a
map which preserves composition not exactly, but only up to
\emph{constraint isomorphisms} $F(g) \circ F(f) \cong F(g\circ f)$ and
$\Id\cong F(\Id)$ (which are then required to satisfy standard
\emph{coherence axioms}).

\begin{thm}[\cite{hovey:modelcats}]
  There are pseudofunctors
  \begin{align*}
    \bL\maps \cModell &\too \cCat\\
    \bR\maps \cModelr &\too \cCat
  \end{align*}
  which take a model category \sC\ to $\Ho(\sC)$ and a left or right
  Quillen functor to its left or right derived functor, respectively.
\end{thm}
\begin{proof}
  Consider \bL; of course \bR\ is dual.  We have already defined the
  image of each model category and each left Quillen functor.  If
  $f,g\maps \sC\toto\sD$ are left Quillen and $\al\maps f\to g$ is a
  natural transformation, then the image of $\al$ under \bL\ is
  defined to be the natural transformation $\bL f\to\bL g$ whose
  components are represented by $\al_{QX}\maps fQX\to gQX$.  We refer
  to this as the \textbf{derived natural transformation} of $\al$.  This
  operation clearly preserves composites of natural transformations.
  The pseudofunctor composition constraint $\bL g\circ \bL f
  \too[\iso] \bL(g\circ f)$ is represented by the natural
  transformation
  \begin{equation}
    \xymatrix@C=3pc{g Q f Q \ar[r]^{g\pi f Q} & g f Q}\label{eq:psfr-coh}
  \end{equation}
  where $\pi\maps Q\to \Id$ is a natural weak equivalence relating $Q$
  to the identity.  Since $f$ preserves cofibrant objects, $\pi f Q$
  is a weak equivalence between cofibrant objects,
  so~\eqref{eq:psfr-coh} is also a weak equivalence and thus
  represents an isomorphism in the homotopy category.  The unit
  isomorphism $\bL (\Id) \too[\iso] \Id$ is simply represented by $\pi$
  itself, and the axioms of a pseudofunctor follow by naturality of
  $\pi$.
\end{proof}

As usual, the existence of a functor implies the automatic
preservation of any structure that can be defined in the relevant sort
of category.  In this case, that means any categorical structure that
can be ``internalized'' to any 2-category.

\begin{eg}\label{eg:der-2cat-adjn}
  An \emph{adjunction} $f\adj g$ in a 2-category consists of morphisms
  $f\maps C\to D$ and $g\maps D\to C$ and 2-cells $\eta\maps \id_C \to
  gf$ and $\ep\maps fg\to \id_D$ satisfying the usual triangle
  identities.  An adjunction in \cCat\ is just an adjunction in the
  usual sense.

  Since adjunctions are defined purely 2-categorically, they are
  preserved by any pseudofunctor.  Thus, if $f\adj g$ is an adjunction
  between model categories in which both $f$ and $g$ are left
  Quillen, then we also have $\bL f\adj \bL g$.
\end{eg}

\begin{eg}\label{eg:2cat-der-mates}
  Of greatest interest to us is that \emph{mates} can be defined
  internal to any 2-category.  The definitions are the same as those
  given in \S\ref{sec:mates}: simply replace ``functor'' by
  ``morphism'' and ``natural transformation'' by ``2-cell.''

  Since the definition is purely 2-categorical, such mates are also
  preserved by any pseudofunctor.  Thus, in any of the examples given
  in \S\ref{sec:mates}, if all the categories are model categories and
  all the functors involved are (say) left Quillen, then \emph{the mate of a
    derived natural transformation is the same as the derived natural
    transformation of a mate}.  In particular, since pseudofunctors
  take 2-cell isomorphisms to 2-cell isomorphisms, if the mate of a
  given transformation is an isomorphism, then so is the mate of its
  derived transformation.

  For example, given a square~\eqref{eq:bc-down} which satisfies the
  Beck-Chevalley condition on the point-set level, and in which the
  functors $f^*$, $g^*$, $h^*$, $k^*$, $f_!$, and $g_!$ are all left
  Quillen, it follows that the square also satisfies the
  Beck-Chevalley condition at the derived level (i.e.\ the canonical
  transformation $\bL f_! \circ \bL h^* \to \bL k^* \circ \bL g_!$ is
  an isomorphism).
\end{eg}

This is a very appealing formal setup---we have not just one but two
functors---but unfortunately it is not all that useful in practice.
It is certainly useful to know that passage to derived functors
\emph{of the same handedness} preserves composition (the pseudofunctor
constraint $\bL g\circ \bL f \iso \bL(g\circ f)$), but it turns out
that one almost never encounters adjunctions in \cModell\ or \cModelr.
Much more common are, of course, Quillen adjunctions, in which the
\emph{left} adjoint is \emph{left} Quillen and (equivalently) the
\emph{right} adjoint is \emph{right} Quillen.  It is well-known that
any Quillen adjunction $f\adj g$ has a derived adjunction $\bL f\adj
\bR g$, but since $f$ and $g$ live in different 2-categories this does
not follow from pseudofunctoriality as in \autoref{eg:der-2cat-adjn}.
However, functoriality is such a useful type of framework that it is
natural to ask whether there is some other type of ``functor''
which can serve to relate left and right derived functors.

In the rest of the paper we give an affirmative answer to this
question.  However, such an answer must move beyond
2-categories; it is impossible to have a 2-category \cK\ in which
Quillen adjunctions are internal adjunctions and which admits a
pseudofunctor $\cK\to\cCat$ combining \bL\ and \bR.  For if so, then
as in \autoref{eg:2cat-der-mates}, this would imply that any
Beck-Chevalley condition that holds on the point-set level would
remain true at the derived level.  But this is known to be
false; see, for instance,~\cite[Counterexample 0.0.1]{maysig:pht}
(which we repeat below as \autoref{thm:not-bc}).

\section{Double categories}
\label{sec:dblcats}

Roughly speaking, the problem we encountered in the previous section
is that the composite of a left Quillen functor and a right
Quillen functor need not be any sort of Quillen functor.  It turns out,
however, that there is a well-known structure which precisely allows
us to speak about 2-cells such as the $\eta$ and $\ep$ in a Quillen
adjunction, but without necessarily being able to actually ``compose''
$f$ and $g$.  This structure is called a \textbf{double category}.

Double categories are a fundamental categorical notion, like
2-categories (although historically, they have received less
attention).  As such, they can be seen from many different viewpoints
and play many different roles in mathematics (which also leads to many
variants of the definition).  Double categories were originally
introduced by Ehresmann~\cite{ehresmann:cat-str}; a good reference
with a point of view similar to ours is~\cite{ks:r2cats}.

A double category \uK\ consists of the following data.  First of all,
we have \emph{two} categories with the same set of objects (or
\emph{0-cells}).  We distinguish between the two types of morphisms
(or \emph{1-cells}) by calling one of them \emph{vertical} and one of
them \emph{horizontal}, and usually drawing them accordingly.  In
addition, there are \emph{squares} (or \emph{2-cells}) which have the
following shape:
\begin{equation}
  \xymatrix{a \ar[r]^f\ar[d]_h \drtwocell\omit{\alpha}&
  c \ar[d]^k\\
  b \ar[r]_g & d.}\label{eq:2cell}
\end{equation}
Here $a$, $b$, $c$, and $d$ are objects, $f$ and $g$ are horizontal
morphisms, and $h$ and $k$ are vertical morphisms.  We think of such
an $\alpha$ as a morphism from ``the composite $kf$'' to ``the
composite $gh$,'' even though such composites do not actually exist
(since the vertical and horizontal 1-cells live in different
categories).

Finally, we require that the 2-cells can be composed both horizontally
and vertically, forming the morphisms of a category in each direction,
and that these two category structures respect each other and the
given categories of horizontal and vertical 1-cells.  We write
$\alpha\boxbar\beta$ for the horizontal composite of 2-cells
\[\xymatrix{ \ar[r]\ar[d] \drtwocell\omit{\alpha} &
  \ar[r]\ar[d] \drtwocell\omit{\beta} &
  \ar[d]\\
  \ar[r] & \ar[r] & }\]
and $\beta\boxminus\alpha$ for the vertical composite
\[\xymatrix{ \ar[r]\ar[d] \drtwocell\omit{\alpha} &  \ar[d]\\
  \ar[r]\ar[d] \drtwocell\omit{\beta} &  \ar[d]\\
  \ar[r] & .}\]
The compatibility requirement for composition is then that
\[(\alpha\boxbar \beta)\boxminus(\gamma\boxbar \delta) =
(\alpha \boxminus \gamma)\boxbar(\beta\boxminus \delta).\]
Every object $a$ has both a vertical identity $1_a$ and a horizontal
identity $1^a$, every vertical arrow $g\maps a\to b$ has an identity
2-cell
\[\xymatrix{\ar[r]^{1^a}\ar[d]_g \drtwocell\omit{1^g} &  \ar[d]^g\\
  \ar[r]_{1^b} & }\]
every horizontal arrow $f\maps a\to c$ has an identity 2-cell
\[\xymatrix{ \ar[r]^f\ar[d]_{1_a} \drtwocell\omit{1_f} &  \ar[d]^{1_c}\\
  \ar[r]_f & }\]
and the compatibility requirements for units are that
\begin{equation*}
  1^{1_a} = 1_{1^a}, \qquad
  1^g \boxminus 1^f = 1^{gf}, \quad\text{and}\quad
  1_f \boxbar 1_g = 1_{gf}.
\end{equation*}
We will often write identity arrows simply as equalities.

\begin{rmk}
  A more concise definition of a double category is that it is an
  \emph{internal} category in \cCat\ (as contrasted with a 2-category,
  which is a category \emph{enriched} in \cCat).  This definition is
  often convenient for dealing with weak double categories (the
  double-category counterpart of bicategories, or weak 2-categories).
  However, for our purposes this approach merely muddies the water,
  since it breaks the symmetry between the horizontal and vertical
  directions.
\end{rmk}

The following examples are fundamental.

\begin{eg}
  There is a double category \uCat\ whose objects are categories,
  whose vertical and horizontal 1-cells are functors, and whose
  2-cells of the form~(\ref{eq:2cell}) are natural transformations
  $\alpha\maps kf\to gh$.
\end{eg}

\begin{eg}
  A similar double category can be constructed with any 2-category
  \cK\ replacing \cCat; we call this the double category $\uSq(\cK)$
  of \textbf{squares} in \cK.  Ehresmann, who first defined it, called
  it the double category of \emph{quintets} in \cK, since a 2-cell in
  $\uSq(\cK)$ is defined by a quintet $(f,g,h,k,\alpha)$ where
  $\alpha\maps kf\to gh$ is a 2-cell in \cK.
\end{eg}

\begin{eg}
  Any ordinary category \bC\ can be regarded as a 2-category with only
  identity 2-cells, so we thereby obtain a double category
  $\uSq(\bC)$ of commutative squares in \bC.
\end{eg}

In double categories of the form $\uSq(\cK)$, the vertical and
horizontal 1-cells are the same.  Of course, the reason for
considering double categories instead of 2-categories is that the two
can also be different.  The following example is the one in which we
are most interested.

\begin{eg}
  There is a double category \uModel\ whose objects are model
  categories, whose vertical arrows are left Quillen functors, whose
  horizontal arrows are right Quillen functors, and whose 2-cells of
  the form~(\ref{eq:2cell}) are arbitrary natural transformations
  $\alpha\maps kf\to gh$.  (The reason for these particular choices
  will become clear in \S\ref{sec:derived-functors}.)
\end{eg}

Any double category has two underlying 2-categories with the same
objects, called its \textbf{horizontal 2-category} $\cH(\uK)$ and its
\textbf{vertical 2-category} $\cV(\uK)$.  The morphisms of $\cH(\uK)$ are
the horizontal 1-cells of \uK, and its 2-cells are the squares in \uK\
of the form
\[\xymatrix{ \ar[r]\ar@{=}[d] \drtwocell\omit &  \ar@{=}[d]\\
  \ar[r] & }\]
which we call \textbf{h-globular}.  Dually, $\cV(\uK)$ is composed of
the objects, vertical 1-cells, and \textbf{v-globular} squares in \uK.

\begin{eg}
  Of course, for any \cK\ we have $\cH(\uSq(\cK))\iso \cK$ and
  $\cV(\uSq(\cK))\iso \cK$.  More interestingly, we have $\cH(\uModel)\iso
  \cModelr$ and $\cV(\uModel)\iso \cModell$.
\end{eg}

Any double category \uK\ has three \emph{opposites}, obtained by
reversing it horizontally, vertically, or both.  It also has a
\emph{transpose} $\uK\tr$ obtained by switching the vertical and
horizontal arrows.



We end this section by mentioning one further class of examples.  We
will make no real use of these in this paper, but they are worth
thinking about for purposes of comparison and intuition.

\begin{eg}\label{eg:moncat}
  There is a double category \uMonCat\ whose objects are monoidal
  categories, whose horizontal arrows are \emph{lax} monoidal
  functors, and whose vertical arrows are \emph{colax} monoidal
  functors.  A 2-cell
  \[\xymatrix{a \ar[r]^f\ar[d]_h \drtwocell\omit{\alpha}&
    c \ar[d]^k\\
    b \ar[r]_g & d}\]
  is a natural transformation $\alpha\maps kf\to gh$ such that the
  following diagrams commute:
  \[\vcenter{\xymatrix@R=1.5pc@C=1pc{ & k(fx\ten fy) \ar[dl]_{k_\ten} \ar[dr]^{k(f_\ten)}\\
      kfx \ten kfy \ar[d]_{\alpha\ten\alpha} &&
      kf(x\ten y) \ar[d]^\alpha \\
      ghx \ten ghy \ar[dr]_{g_\ten} &&
      gh(x\ten y) \ar[dl]^{g(h_\ten)}\\
      & g(hx \ten hy)}}
  \qquad\text{and}\qquad
  \vcenter{\xymatrix@R=1.5pc@C=1pc{ & k(I_c) \ar[dl]_{k_I} \ar[dr]^{k(f_I)}\\
      \quad I_d\quad \ar@{=}[d] &&
      kf(I_a) \ar[d]^\alpha \\
      \quad I_d\quad \ar[dr]_{g_I} &&
      gh(I_a) \ar[dl]^{g(h_I)}\\
      & g(I_b)}}\]

  The horizontal 2-category $\cH(\uMonCat)$ of \uMonCat\ is the
  2-category $\cMonCat_\ell$ of monoidal categories, lax monoidal
  functors, and monoidal transformations, and dually for
  $\cV(\uMonCat)$.  More generally, we have a double category \bbTAlg\
  of $T$-algebras, lax and colax $T$-morphisms, and generalized
  $T$-transformations for any 2-monad $T$.  (A \emph{2-monad} is a
  monad on a 2-category, for which we can define general notions of
  lax and colax morphisms of algebras; see~\cite{bkp:2dmonads}.)
  These double categories were apparently first considered
  in~\cite{gp:double-adjoints}.
\end{eg}

\begin{rmk}
  In \uModel\ no axioms are imposed on the 2-cells, whereas in
  \uMonCat\ there is a compatibility requirement with the structure
  of the 1-cells.  Nevertheless, there is a connection between the
  two.  One can ``algebraicize'' parts of the definition of model
  category so that ``algebraicized'' left and right Quillen functors
  become colax and lax morphisms for a 2-monad, respectively;
  see~\cite{gt:nwfs,garner:soa,riehl:nwfs-model}.
\end{rmk}

\section{Companions and conjoints}
\label{sec:compconj}

Recall that our goal in introducing double categories was to find an
abstract framework in which to express the adjointness between a left
Quillen functor and a right Quillen functor.  Inspecting the
definition of \uModel, we immediately see how to write this down.
The following terminology is due to~\cite{dpp:spans}
(in~\cite{gp:double-adjoints} it was called an \emph{orthogonal
  adjunction}).

\begin{defn}
  A \textbf{conjunction} in a double category \uK\ consists of
  a vertical 1-cell $f\maps a\to b$, a horizontal 1-cell $g\maps b\to
  a$, and 2-cells
  \[\xymatrix{a \ar@{=}[r]\ar[d]_f \drtwocell\omit{\eta} & a \ar@{=}[d]\\
    b\ar[r]_g & a} \qquad\text{and}\qquad
  \xymatrix{b \ar[r]^g\ar@{=}[d] \drtwocell\omit{\ep} & a \ar[d]^f\\
    b\ar@{=}[r] & b}
  \]
  (the \textbf{unit} and \textbf{counit}) such that $\ep\boxbar\eta =
  1_g$ and $\ep\boxminus\eta = 1^f$.  We say that $f$ is the
  \textbf{left conjoint} and $g$ is the \textbf{right conjoint}, and
  write $f\conj g$.
\end{defn}

\begin{eg}
  A conjunction in \uCat\ is simply an ordinary adjunction.
  Likewise, a conjunction in $\uSq(\cK)$ is simply an ordinary
  internal adjunction in \cK.
\end{eg}

\begin{eg}\label{eg:qadj-as-conj}
  A conjunction in \uModel\ is an adjunction in which the left adjoint
  is left Quillen and the right adjoint is right Quillen---in other
  words, a Quillen adjunction.
\end{eg}

\begin{eg}
  A conjunction in \bbTAlg\ (such as \uMonCat) is precisely a
  \emph{doctrinal adjunction} as studied in~\cite{kelly:doc-adjn}.
  This is an adjunction between $T$-algebras in which the left adjoint
  is colax and the right adjoint is lax, and the colax and lax
  structure maps are mates under the
  adjunction.
\end{eg}

Interpreting conjunctions in the horizontally-opposite double category
(or, equivalently, the vertically-opposite one), we obtain a different
useful notion.

\begin{defn}
  A \textbf{companion pair} in a double category \uK\ consists
  of a vertical 1-cell $f\maps a\to b$, a horizontal 1-cell $f'\maps
  a\to b$, and 2-cells
  \[\xymatrix{a \ar[r]^{f'}\ar[d]_f \drtwocell\omit{\ph} & b \ar@{=}[d]\\
    b\ar@{=}[r] & a} \qquad\text{and}\qquad
  \xymatrix{a \ar@{=}[r]\ar@{=}[d] \drtwocell\omit{\psi} & a \ar[d]^f\\
    a\ar[r]_{f'} & b}
  \]
  such that $\ph\boxminus\psi = 1^f$ and $\psi\boxbar\ph = 1_{f'}$.
  We say that $f'$ is the \textbf{(horizontal) companion} of $f$ and
  that $f$ is the \textbf{(vertical) companion} of $f'$, and write
  $f\comp f'$.
\end{defn}

\begin{eg}
  In \uCat\ or $\uSq(\cK)$, every 1-cell has a companion, namely
  itself; $\ph$ and $\psi$ are both identities.  More generally, a
  companion pair in $\uSq(\cK)$ is precisely a natural isomorphism
  between two parallel morphisms $f,g\maps a\toto b$ in \cK.
\end{eg}

Thus, from the double-categorical perspective, ``adjunctions are dual
to natural isomorphisms.''

\begin{eg}
  A 1-cell in \uModel\ has a companion just when it both left and
  right Quillen.
\end{eg}

\begin{eg}
  A 1-cell in \uMonCat\ has a companion just when it is a
  \emph{strong} monoidal functor.  For the 2-cells $\ph$ and $\psi$
  show that $f$ and $f'$ are isomorphic as ordinary functors, and then
  the hexagon axioms in the definition of a 2-cell in \uMonCat\ imply
  that the lax structure maps of $f'$ are inverses to the colax
  structure maps of $f$, so that both are strong.  An analogous
  statement is true in any \bbTAlg.
\end{eg}

Motivated by these examples, we say that a 1-cell in a general double
category is \textbf{strong} if it has a companion.

Companions and conjoints in a double category have most of the good
properties of adjunctions in a 2-category.  For instance, they are
unique up to unique isomorphism when they exist, and are preserved
under composition.

\begin{prop}\label{thm:comp-uniq}
  If $f'$ and $f''$ are both horizontal companions of $f$, then there
  is a canonical isomorphism $f'\iso f''$ in $\cH(\uK)$, and
  similarly for vertical companions.
\end{prop}
\begin{proof}
  An isomorphism is given by the following composite.
  \[\begin{array}{c}
    \xymatrix{ \ar@{=}[r]\ar@{=}[d] \drtwocell\omit &
      \ar[r]^{f'}\ar[d]^f \drtwocell\omit &
      \ar@{=}[d]\\
      \ar[r]_{f''} & \ar@{=}[r] & }
  \end{array}\]
  Its inverse is the obvious dual construction.
\end{proof}

\begin{prop}\label{thm:comp-compose}
  If $f$ and $h$ have companions $f'$ and $h'$, then $h' f'$ is a
  companion of $h f$.
\end{prop}
\begin{proof}
  It is straightforward to compose the 2-cells defining the companion
  pairs $f\comp f'$ and $h\comp h'$ to produce a companion pair $h
  f\comp h'f'$.
\end{proof}

By duality, we have the corresponding results for conjunctions.

\begin{prop}
  If $g$ and $g'$ are both conjoints of $f$, then there is a canonical
  globular isomorphism $g\iso g'$.
\end{prop}

\begin{prop}
  If $f$ and $h$ have conjoints $g$ and $k$, respectively, then $gh$
  is a conjoint of $hf$.
\end{prop}

The most important property of companions and conjunctions for our
purposes, however, is that they also have an associated mate
correspondence.  We begin with mates for companions.

\begin{prop}\label{companion-mates}
  If $f$ and $g$ have horizontal companions $f'$ and $g'$, then there is a
  canonical isomorphism
  \begin{equation}
    \cV(\uK)(f,g) \iso \cH(\uK)(f',g').\label{eq:companion-mates-i}
  \end{equation}
  More generally, for any $i,j,m,n$ there is a bijection between
  2-cells of the following shapes:
  \begin{equation}
    \begin{array}{c}
      \xymatrix{ \ar[r]^j\ar[d]_{fi} \drtwocell\omit{\alpha} &  \ar[d]^{mg}\\
        \ar[r]_n & }
    \end{array}
    \qquad\text{and}\qquad
    \begin{array}{c}
      \xymatrix{ \ar[r]^{g'j}\ar[d]_i \drtwocell\omit{\beta} &  \ar[d]^m\\
        \ar[r]_{nf'} & .}
    \end{array}\label{eq:companion-mates-ii}
  \end{equation}
  We say that a pair of 2-cells which correspond
  under~\eqref{eq:companion-mates-ii} are \textbf{mates}.
\end{prop}
\begin{proof}
  The bijection is given by the following correspondences.
  \begin{equation}
  \begin{array}{c}
      \xymatrix{ \ar[r]^j\ar[d]_{fi} \drtwocell\omit{\alpha} &
        \ar[d]^{mg}\\
        \ar[r]_n & }
    \end{array} \quad\xymatrix{\ar@{|->}[r] & }\quad
    \begin{array}{c}
      \xymatrix{
        \ar@{=}[r] \ar[d]_i \drtwocell\omit{1^i} & \ar[r]^j\ar[d]^i \ddrtwocell\omit{\alpha}&
        \ar[d]_g \ar[r]^{g'} \drtwocell\omit{\ph} & \ar@{=}[d]\\
        \ar@{=}[d] \ar@{=}[r] \drtwocell\omit{\psi} &
        \ar[d]^f &  \ar[d]_m \ar@{=}[r] \drtwocell\omit{1^m} & \ar[d]^m \\
        \ar[r]_{f'} & \ar[r]_n & \ar@{=}[r] & }
    \end{array}\label{eq:mate-bijection-i}
  \end{equation}
  \begin{equation}
    \begin{array}{c}
    \xymatrix{
      \ar@{=}[d] \ar[r]^j \drtwocell\omit{1_j}&
      \ar@{=}[d] \ar@{=}[r] \drtwocell\omit{\psi} & \ar[d]^g\\
      \ar[r]_j\ar[d]_i \drrtwocell\omit{\beta} &
      \ar[r]_{g'} & \ar[d]^m\\
      \ar[r]^{f'} \ar[d]_f \drtwocell\omit{\ph} &
      \ar[r]^n \ar@{=}[d] \drtwocell\omit{1_n} &
      \ar@{=}[d] \\
      \ar@{=}[r] &\ar[r]_n& }
    \end{array} \quad\xymatrix{& \ar@{|->}[l] }\quad
    \begin{array}{c}
      \xymatrix{ \ar[r]^{g'j}\ar[d]_i \drtwocell\omit{\beta} &  \ar[d]^m\\
        \ar[r]_{nf'} & .}
    \end{array}
  \end{equation}
\end{proof}

The correspondence~\eqref{eq:companion-mates-i}
is preserved by composition, so we have a 2-category
$\cStr(\uK)$ whose 0-cells are the 0-cells of \uK, whose 1-cells
are companion pairs in \uK, and whose 2-cells are mate-pairs of
globular 2-cells.  We also have canonical 2-functors
\begin{align*}
  \cStr(\uK) &\too \cV(\uK)\\
  \cStr(\uK) &\too \cH(\uK)
\end{align*}
which are full and faithful on hom-categories.

\begin{egs}
  $\cStr(\uSq(\cK))$ is not quite the same as \cK; its morphisms are
  pairs of parallel morphisms in \cK\ with an isomorphism between
  them.  However, it is ``biequivalent'' to \cK\ (this is the most
  general sort of equivalence between 2-categories).

  Similarly, $\cStr(\uModel)$ is biequivalent to the 2-category of
  model categories and functors which are both left and right
  Quillen, and $\cStr(\uMonCat)$ is biequivalent to the 2-category
  of monoidal categories and strong monoidal functors.
\end{egs}

Dualizing this correspondence, we immediately obtain the mate
correspondence for conjunctions.

\begin{prop}\label{conjunction-mates}
  If $f\conj g$ and $h\conj k$ where $f,h\maps a\to b$, then we have a
  natural isomorphism
  \begin{equation}
    \cV(\uK)(f,h) \iso \cH(\uK)(k,g),\label{eq:conjoint-mates-i}
  \end{equation}
  under which isomorphisms $f\iso h$ correspond to isomorphisms $k\iso
  g$.  More generally, for any $i,j,m,n$ there is a bijection between
  2-cells of the following shapes:
  \begin{equation}
    \begin{array}{c}
      \xymatrix{ \ar[r]^i\ar[d]_{mh} \drtwocell\omit & \ar[d]^{fj}\\
        \ar[r]_{n} & }
    \end{array}
    \qquad\text{and}\qquad
    \begin{array}{c}
      \xymatrix{ \ar[r]^{ik}\ar[d]_m \drtwocell\omit &  \ar[d]^j\\
        \ar[r]_{gn} & }
    \end{array}.\label{eq:conjoint-mates-ii}
  \end{equation}
  We say that a pair of 2-cells which correspond
  under~\eqref{eq:conjoint-mates-ii} are \textbf{mates}.
\end{prop}

As with companions, we obtain a 2-category $\cConj(\uK)$ whose
objects are those of \uK, whose 1-cells are the conjunctions in \uK,
and whose 2-cells are the mate-pairs of globular 2-cells in \uK.
Note, though, that to define $\cConj(\uK)$ we must choose whether to consider
a conjunction as pointing in the direction of the left conjoint or the
right conjoint; it is precisely this arbitrariness which the
double-categorical context avoids.

\begin{egs}\label{eg:2cats-of-conj}
  Of course, $\cConj(\uCat)$ is the usual 2-category of categories and
  adjunctions, while $\cConj(\uModel)$ is the usual 2-category of
  model categories and Quillen adjunctions.
\end{egs}

\begin{rmk}
  Mates for conjunctions are clearly analogous to mates for
  adjunctions in a 2-category.  The mate correspondence for
  companion pairs also has an analogue in a 2-category, though it is
  too obvious to require comment (or a name): it simply says that if
  $f\iso f'$ and $g\iso g'$, then there is a bijection between 2-cells
  $f\to g$ and $f'\to g'$.
\end{rmk}

As an immediate application of mates, we show that companion pairs
``mediate'' between adjunctions and conjunctions.

\begin{prop}\label{compconj-mediate}
  Let $f\maps a\to b$ be a vertical 1-cell in \uK\ and let $f'\maps
  a\to b$ and $g\maps b\to a$ be horizontal 1-cells.  Then any two of
  the following statements imply the third.
  \begin{enumerate}
  \item $f'$ is a horizontal companion of $f$.\label{item:docadj-companion}
  \item $g$ is a right conjoint of $f$.\label{item:docadj-conjoint}
  \item $g$ is a right adjoint of $f'$ in $\cH(\uK)$.\label{item:docadj-adjoint}
  \end{enumerate}
  More precisely, any companion pair $f\comp f'$ and conjunction
  $f\conj g$ determine a unique horizontal adjunction $f'\adj g$, and
  similarly in the other cases.
\end{prop}
\begin{proof}
  Assuming~\ref{item:docadj-companion}, the correspondence of
  \autoref{companion-mates} transforms a unit and counit for a
  conjunction $f \conj g$ into a unit and counit for a horizontal
  adjunction $f'\adj g$, and vice versa.  The other cases are similar.
\end{proof}

\begin{rmk}
  In \bbTAlg, this implies part of one of the main results
  of~\cite{kelly:doc-adjn}: in a doctrinal adjunction, the left
  adjoint is a strong $T$-morphism precisely when the adjunction is an
  adjunction in the 2-category $\cTAlg_\ell$.
\end{rmk}

\begin{rmk}
  We believe that the notions of companion pair and conjunction are as
  central to the theory of double categories as the notions of
  equivalence and adjunction are to the theory of 2-categories.  It is
  thus surprising that they seem only recently to have been isolated
  in the present form.
  The basic ideas, however, have been around a long time.  For
  instance, a \emph{folding} or \emph{connection pair} on a double
  category, as considered
  in~\cite{bs:dblgpd-xedmod,bm:dbl-thin-conn,fiore:pscat}, can be
  defined as a strictly functorial choice of a companion for each
  vertical 1-cell.  Since companions are always pseudofunctorial
  (\autoref{thm:comp-compose}), an arbitrary choice of companions for
  each vertical 1-cell is the same as a \emph{pseudo-folding} in the
  sense of~\cite{fiore:pscat}.  The \emph{framed bicategories}
  of~\cite{shulman:frbi} are (pseudo) double categories in which every
  vertical 1-cell has both a horizontal companion and a right
  conjoint.
\end{rmk}

\section{Double pseudofunctors}
\label{sec:double-psfrs}

The way forward should now be clear: we aim to show that passage to
homotopy categories and derived functors is a functor
$\uModel\to\uCat$.  However, as in the 2-categorical case, it can only
be expected to be a \emph{pseudofunctor}, i.e.\ to preserve
composition and identities up to coherent isomorphism.  Strict
functors of double categories are easy to define, and functors which
are pseudo in one direction and strict in the other also appear in the
literature under the name of \emph{pseudo double functor} (see, for
instance,~\cite{gp:double-limits}), but we require functors which are
pseudo in both directions.  We now define these precisely under the
name of \emph{double pseudofunctors}; the reader who is uninterested
in the details may skim this section.

\begin{defn}\label{defn:dbl-psfr}
  Let \uK\ and \uL\ be double categories.  A \textbf{double
    pseudofunctor} $F\maps \uK\to\uL$ consists of the following
  structure and properties.
  \begin{enumerate}
  \item Functions from the objects, vertical 1-cells, horizontal
    1-cells, and 2-cells of \uK\ to those of \uL, preserving
    sources, targets, and boundaries.
  \item For each object $a$ of \uK, 2-cells
    \[
    \begin{array}{c}
      \xymatrix{Fa \ar[r]^{F(1^a)}\ar@{=}[d]_{1_{Fa}}
        \drtwocell\omit{F_a} &
        Fa \ar@{=}[d]^{1_{Fa}}\\
        Fa \ar@{=}[r]_{1^{Fa}} & Fa}
    \end{array}
    \qquad\text{and}\qquad
    \begin{array}{c}
      \xymatrix{Fa \ar@{=}[r]^{1^{Fa}}\ar@{=}[d]_{1_{Fa}}
        \drtwocell\omit{F^a} &
        Fa \ar[d]^{F(1_{a})}\\
        Fa \ar@{=}[r]_{1^{Fa}} & Fa}
    \end{array}
    \]
    in \uL, of which the first is an h-globular isomorphism and the
    second a v-globular isomorphism.
  \item For each composable pair $a\too[f] b\too[g] c$ of vertical
    1-cells in \uK, a v-globular isomorphism
    \[\xymatrix{Fa \ar@{=}[r]\ar[d]_{Ff} \ddrtwocell\omit{\;F^{gf}} &
      Fa \ar[dd]^{F(gf)}\\
      Fb \ar[d]_{Fg} & \\
      Fc \ar@{=}[r] & Fc.}\]
  \item For each composable pair $a\too[h] b\too[k] c$ of horizontal
    1-cells in \uK, an h-globular isomorphism 
    \[\xymatrix{Fa \ar[rr]^{F(kh)}\ar@{=}[d]
      \drrtwocell\omit{\quad F_{kh}} &&
      Fc \ar@{=}[d]\\
      Fa \ar[r]_{Fh} & Fb \ar[r]_{Fk} & Fc.}\]
  \item The following coherence axioms hold (the usual coherence
    axioms for a pseudofunctor in both directions).\label{item:psfr-ax}
    \begin{align*}
      F^{h(gf)} \boxbar \left(1^{Fh} \boxminus F^{gf}\right)
      &= F^{(hg)f} \boxbar \left(F^{hg} \boxminus 1^{Ff}\right)\\
      F^b \boxminus 1^{Ff} &= F^{1_b f}\\
      1^{Ff} \boxminus F^a &= F^{f 1_a}\\
      F_{h(gf)} \boxminus \left(F_{gf} \boxbar 1_{Fh}\right)
      &= F_{(hg)f} \boxminus \left(1_{Ff} \boxbar F_{hg}\right)\\
      F_a \boxbar 1_{Ff} &= F_{f 1_a}\\
      1_{Ff} \boxbar F_b &= F_{1_b f}.
    \end{align*}
  \item The ``double naturality'' axioms displayed in figure~\ref{fig:dblnat} hold,
    \begin{figure}
    \begin{align}\label{eq:dblnat-1}
      \begin{array}{c}
        \xymatrix{
          \ar[rr]^{F(gf)} \ar@{=}[d] \drrtwocell\omit{\quad F_{gf}} && \ar@{=}[d]\\
          \ar[d]_{Fu}\ar[r]\drtwocell\omit{F\alpha} &
          \ar[d]|{Fv}\ar[r]\drtwocell\omit{F\beta} & \ar[d]^{Fw}\\
          \ar[r]_{Fh} & \ar[r]_{Fk} &}
      \end{array}
      &=
      \begin{array}{c}
        \xymatrix{
          \ar[rr]^{F(gf)}\ar[dd]_{Fu}
          \ar@{}[ddrr]_(.3){}="d1" \ar@{=>}"d1"+/^ 0.5pc/;"d1"+/_ 0.5pc/ ^{F(\alpha\boxbar\beta)}
          &&
          \ar[dd]^{Fw}\\
          & \\
          \ar[rr] \ar@{=}[d] \drrtwocell\omit{\quad F_{kh}} && \ar@{=}[d]\\
          \ar[r]_{Fh} & \ar[r]_{Fk} &}
      \end{array}\\
      \label{eq:dblnat-2}
      \begin{array}{c}
        \xymatrix{
          Fa \ar@{=}[r]\ar@{=}[d]
          \ar@{}[dr]_(.4){}="d1" \ar@{=>}"d1"+/^ 0.5pc/;"d1"+/_ 0.5pc/ ^{F_a\inv}
          & Fa \ar@{=}[d]\\
          Fa \ar[r]\ar[d]_{Ff}
          \ar@{}[dr]_(.3){}="d2" \ar@{=>}"d2"+/^ 0.5pc/;"d2"+/_ 0.5pc/ ^{F(1^f)}
          & Fa \ar[d]^{Ff}\\
          Fa \ar[r]\ar@{=}[d]
          \ar@{}[dr]_{}="d3" \ar@{=>}"d3"+/^ 0.5pc/;"d3"+/_ 0.5pc/ ^{F_a}
          & Fa \ar@{=}[d]\\
          Fa \ar@{=}[r] & Fa}
      \end{array}
      &=
      \begin{array}{c}
        \xymatrix{Fa \ar@{=}[r]\ar[d]_{Ff}
          \ar@{}[dr]_(.35){}="d2" \ar@{=>}"d2"+/^ 0.5pc/;"d2"+/_ 0.5pc/ ^{1^{Ff}} &
          Fa\ar[d]^{Ff}\\
          Fa\ar@{=}[r] & Fa}
      \end{array}
    \end{align}
    \caption{The horizontal double naturality axioms}
    \label{fig:dblnat}
  \end{figure}
    as do their transposes involving $F^{gf}$ and $F^a$.
  \end{enumerate}
\end{defn}

Note that in general, a double pseudofunctor does not preserve
globularity of 2-cells, since it does not preserve either vertical or
horizontal identities strictly.  However, any h-globular 2-cell
\[\xymatrix{a \ar[r]^f\ar@{=}[d] \drtwocell\omit{\alpha} & b\ar@{=}[d]\\
  a\ar[r]_g & b}\]
in \uK\ gives rise to a canonical h-globular 2-cell
\[\xymatrix{Fa\ar@{=}[d] \ar@{=}[r]
  \ar@{}[dr]_(.4){}="a" \ar@{=>}"a"+/^ 0.5pc/;"a"+/_ 0.5pc/ ^{F^a} &
  Fa\ar[r]^{Ff}\ar[d]
  \ar@{}[dr]_(.4){}="c" \ar@{=>}"c"+/^ 0.5pc/;"c"+/_ 0.5pc/ ^{F(\alpha)} &
  Fb\ar[d] \ar@{=}[r]
  \ar@{}[dr]_(.3){}="b" \ar@{=>}"b"+/ur 0.5pc/;"b"+/dl 0.5pc/ ^{(F^a)\inv} & Fb\ar@{=}[d]\\
  Fa\ar@{=}[r] & Fa\ar[r]_{Fg} & Fb\ar@{=}[r] & Fb}\]
in \uL, which we denote $\cH F (\alpha)$.  It is easy to check that
this defines an ordinary pseudofunctor $\cH F\maps
\cH\uK\to\cH\uL$.  Similarly, we have $\cV F\maps
\cV\uK\to\cV\uL$.

\begin{egs}
  An ordinary pseudofunctor $F\maps \cK\to\cL$ gives rise to a
  double pseudofunctor $\uSq(F)\maps
  \uSq(\cK)\to\uSq(\cL)$ in a fairly straightforward way.
  The only wrinkle is that if $\alpha\maps kf\to gh$ is a 2-cell in
  $\uSq(\cK)$, we must compose $F\alpha$ with the constraints of
  $F$ on either side to obtain a 2-cell in $\uSq(\cL)$.

  In particular, if $F\maps \cCat\to\cCat$ is a pseudofunctor, we
  obtain a double pseudofunctor $\uSq(F)\maps \uCat\to\uCat$, and some
  of the double pseudofunctors obtained in this way also give
  endofunctors of \uModel.  For instance, there is a double
  pseudofunctor $\uModel\to\uModel$ which takes a model category \sC\
  to its pointed variant $\sC_*$; see~\cite[1.1.8,
  1.3.5]{hovey:modelcats}.
\end{egs}

\begin{eg}
  Recall that $\uK\tr$ denotes the transpose of a double category, in
  which the vertical and horizontal arrows are interchanged.  We then
  have a double pseudofunctor $\uModel\tr\to\uModel$ which takes
  \sC\ to $\sC\op$.
\end{eg}

\begin{eg}
  Recall that for any ordinary category \bC\ we have a double category
  $\uSq(\bC)$ of commutative squares in \bC.  If we restrict the
  squares in $\uSq(\bC)$ to a subclass $\cA$ of commutative squares
  which are closed under composition and identities (such as all
  pullback squares), we obtain a smaller double category
  $\uSq(\bC;\cA)$.  Then for any 2-category \cK, a double
  pseudofunctor $\uSq(\bC;\cA)\to \uSq(\cK)$ is essentially the same
  as a \emph{lower e-functor} relative to \cA, in the sense
  of~\cite[\S4.1]{dv:cross-functors}.  \emph{Upper e-functors} and
  \emph{e$^*$ and e!\ contradirectional functors} are defined by
  applying appropriate types of duality to $\uSq(\cK)$.  Finally, a
  \emph{cross functor} is a double pseudofunctor $\uSq(\bC;\cA)\to
  \uCrs(\cK)$, where $\uCrs(\cK)$ is the double category defined as
  follows:
  \begin{blist}
  \item Its objects are the objects of \cK.
  \item Its horizontal 1-cells $A\to B$ are adjunctions $f^*\maps B
    \toot A \spam f_*$ in \cK, where $f_*$ is the right
    adjoint.
  \item Its vertical 1-cells $A\to B$ are adjunctions $f_!\maps A
    \rightleftarrows B \spam f^!$ in \cK, where $f_!$ is the left
    adjoint.
  \item Its 2-cells
    \[\vcenter{\xymatrix@R=3pc@C=3pc{A \ar@<-1mm>[r]_{f_*} \ar@<1mm>[d]^{h_!} \drtwocell\omit &
        B \ar@<-1mm>[l]_{f^*} \ar@<1mm>[d]^{k_!}\\
        C\ar@<-1mm>[r]_{g_*} \ar@<1mm>[u]^{h^!} &
        D \ar@<-1mm>[l]_{g^*} \ar@<1mm>[u]^{k^!}}}
    \]
    are isomorphisms $h_! f^* \iso g^* k_!$.  Any such isomorphism has
    mates $f^*k^! \to h^! g^*$, $k^!g_* \iso f_* h^!$, and $k_!  f_*
    \to g_* h_!$; thus a cross functor has underlying functors of all
    four sorts considered above.
  \end{blist}
\end{eg}

The composite $G\circ F$ of two double pseudofunctors is defined in
an obvious way, with one minor wrinkle: since $G$ need not preserve
the globularity of the constraints for $F$, we need to compose with
the unit constraints of $G$ when defining the constraints of $GF$.
For example, the composition constraint of $GF$ is given by the composite
\[\xymatrix{
  \ar@{=}[r] \ar@{=}[d]
  \ar@{}[dr]_(.3){}="c" \ar@{=>}"c"+/ur 0.5pc/;"c"+/dl 0.5pc/ ^{G^{Fa}} &
  \ar[rr]^{GF(gf)}\ar[d]
  \ar@{}[drr]_(.4){}="x" \ar@{=>}"x"+/ur 0.5pc/;"x"+/dl 0.5pc/ ^{G(F_{gf})}
  && \ar[d]\ar@{=}[r]
  \ar@{}[dr]_(.3){}="d" \ar@{=>}"d"+/ur 0.5pc/;"d"+/dl 0.5pc/ ^{G^{Fc}}&
  \ar@{=}[d] \\
  \ar@{=}[r] \ar@{=}[d]
  \ar@{}[drrrr]_(.4){}="b" \ar@{=>}"b"+/ur 0.5pc/;"b"+/dl 0.5pc/ ^{G_{(Fg)(Ff)}}
  & \ar[rr]  &&
  \ar@{=}[r] & \ar@{=}[d] \\
  \ar[rr]_{GFf} && \ar[rr]_{GFg} & &.}\]
We thereby obtain a category \fDbl\ of double categories and double
pseudofunctors.  The operations \cV\ and \cH\ define functors from
\fDbl\ to the category \ftwoCat\ of 2-categories and pseudofunctors.
In the other direction, \uSq\ defines a functor from \ftwoCat\ to
\fDbl.


The most important observation about double pseudofunctors for our
purposes is that they preserve companions, conjoints, and mates.

\begin{prop}\label{thm:comp-dblpsfr}
  If $f$ has a horizontal companion $f'$ in \uK\ and $F\maps
  \uK\to\uL$ is a double pseudofunctor, then $F(f)$ has a horizontal
  companion $F(f')$.
\end{prop}
\begin{proof}
  We take the defining 2-cells to be
  \[\vcenter{\xymatrix{\ar[r]^{Ff'}\ar[d]_{Ff}
      \ar@{}[dr]_(.4){}="a" \ar@{=>}"a"+/^ 0.5pc/;"a"+/_ 0.5pc/ ^{F\ph} &
      \ar[d] \ar@{=}[r] \drtwocell\omit{\iso} &
      \ar@{=}[d]\\
      \ar@{=}[d]\ar[r] \drtwocell\omit{\iso} &
      \ar@{=}[r] \ar@{=}[d] \ar@{}[dr]|= & \ar@{=}[d] \\
      \ar@{=}[r]_{\phantom{Ff}} & \ar@{=}[r] &}}
  \qquad\text{and}\qquad
  \vcenter{\xymatrix{
      \ar@{}[dr]|= \ar@{=}[d]  \ar@{=}[r] &
      \ar@{=}[r]^{\phantom{Ff}} \ar@{=}[d]\drtwocell\omit{\iso} & \ar@{=}[d]\\
      \ar@{=}[r] \ar@{=}[d] \drtwocell\omit{\iso} &
      \ar[r]\ar[d]
      \ar@{}[dr]_(.4){}="b" \ar@{=>}"b"+/^ 0.5pc/;"b"+/_ 0.5pc/ ^{F\psi} &
      \ar[d]^{Ff}\\
      \ar@{=}[r] & \ar[r]_{Ff'} & }}.
  \]
  Verification of the equations defining a companion pair is
  straightforward using the double naturality axioms
  (Figure~\ref{fig:dblnat}).
\end{proof}

\begin{prop}\label{thm:conj-dblpsfr}
  If $f\conj g$ in \uK\ and $F\maps \uK\to\uL$ is a double
  pseudofunctor, then $F(f)\conj F(g)$.
\end{prop}
\begin{proof}
  By duality.
\end{proof}

\begin{prop}\label{thm:mate-dblpsfr}
  The mate correspondences~\eqref{eq:companion-mates-ii}
  and~\eqref{eq:conjoint-mates-ii} are preserved by double
  pseudofunctors.
\end{prop}
\begin{proof}
  Straightforward verification, again using the double naturality
  axioms and the horizontal and vertical pseudofunctor axioms.
\end{proof}

In particular, this means that we have two additional functors
$\cStr\maps \fDbl \to\ftwoCat$ and $\cConj\maps \fDbl \to\ftwoCat$.
It is of interest to note that \cStr\ is right adjoint to \uSq.

\begin{rmk}
  Double pseudofunctors, as we have defined them, do not seem to
  appear in the literature on double categories.  They can, however,
  be shown to be equivalent to the morphisms of the tricategory
  ${\underline{\underline{\underline{\mathcal{H}\mathit{oriz}}}}}{}_{SH}$
  defined in~\cite[\S1.4]{verity:base-change}, after first identifying
  double categories with a certain strict subclass of the \emph{double
    bicategories} which form the objects of
  ${\underline{\underline{\underline{\mathcal{H}\mathit{oriz}}}}}{}_{SH}$.
\end{rmk}

\section{The double pseudofunctor Ho}
\label{sec:derived-functors}

We are now finally ready to construct the double pseudofunctor
$\Ho\maps \uModel\to\uCat$.  We already know that it should take a
model category to its homotopy category, a left Quillen functor to
its left derived functor, and a right Quillen functor to its right
derived functor, so it remains only to define its action on a 2-cell
\[\xymatrix{\sA \ar[r]^f\ar[d]_h \drtwocell\omit{\alpha}&
  \sC \ar[d]^k\\
  \sB \ar[r]_g & \sD}\]
in \uModel.   We define the derived transformation of such an $\alpha$
to be the transformation
\[\xymatrix{\Ho(\sA) \ar[rr]^{\bR f}\ar[d]_{\bL h}
  \ar@{}[drr]_{}="d1" \ar@{=>}^{\Ho(\alpha)} "d1"+/^ 0.5pc/;"d1"+/_ 0.5pc/
  &&
  \Ho(\sC) \ar[d]^{\bL k}\\
  \Ho(\sB) \ar[rr]_{\bR g} && \Ho(\sD)}\]
represented by the composite of the following zigzag:
\begin{equation}
  kQfR \leftwe kQfQR \too kfQR \too[\alpha] ghQR \too gRhQR \leftwe gRhQ\label{eq:dernat-composite}
\end{equation}
in $\Ho(\sD)$.  Note that the backwards maps are weak equivalences,
hence represent isomorphisms in $\Ho(\sD)$, so this makes sense.
(Recall our convention that the functor $Q$ preserves fibrant
objects.)

\begin{rmk}\label{thm:dernat-simple}
  We can express this more simply as follows.  Assume that $X\in\sC$
  is both cofibrant and fibrant.  Then we have a commutative diagram
  \[\vcenter{\xymatrix@C=1.7pc{
      &
      kQfQRX \ar[r] &
      kfQRX \ar[r]^\alpha &
      ghQRX \ar[r] &
      gRhQRX \ar@{<-}[dr]^-{\sim}\\
      kQfRX \ar@{<-}[ur]^-{\sim} &
      kQfQX \ar[u]_\sim \ar[r] &
      kfQX \ar[r]^\alpha  \ar[u]_\sim &
      ghQX \ar[rr] \ar[u]^\sim &&
      gRhQX \ar[ul]_\sim  & \\
      & kQfX  \ar@{<-}[u]_-{\sim}  \ar[ul]^\sim_{\circledast} \ar[r] &
      kfX \ar[r]^\alpha \ar@{<-}[u]_-{\sim} &
      ghX \ar[r] \ar@{<-}[u]_-{\sim} &
      gRhX \ar@{<-}[ur]_\sim^{\circledast}
    }}
  \]
  in which the zigzag along the top represents $\Ho(\alpha)$ as
  defined above.  The two weak equivalences marked $\circledast$
  represent the canonical isomorphisms $\bR f (X) \iso fX$ and $\bL h
  (X) \iso hX$ when $X$ is fibrant and cofibrant, so this diagram
  shows that modulo these isomorphisms, $\Ho(\alpha)_X$ is
  represented by
  \begin{equation}
    kQ f X \too kf X \too[\alpha] gh X \too gR h X.\label{eq:dernat-composite-short}
  \end{equation}
  This suffices to determine $\Ho(\alpha)$, since every object is
  isomorphic in $\Ho(\sC)$ to a cofibrant and fibrant one.
\end{rmk}

\begin{rmk}
  We could equally well choose to represent $\Ho(\alpha)$ by the
  composite
  \begin{equation}
    kQfR \leftwe kQfRQ \too kfRQ \too[\alpha] ghRQ \too gRhRQ \leftwe
    gRhQ,\label{eq:dernat-composite-alt} 
  \end{equation}
  where we use instead the assumption that $R$ preserves cofibrant
  objects to conclude that the backwards maps are weak equivalences.
  A diagram chase shows that~\eqref{eq:dernat-composite}
  and~\eqref{eq:dernat-composite-alt} represent the same morphism in
  $\Ho(\sD)$.
\end{rmk}

\begin{thm}\label{thm:htpy-dbl-psfr}
  The above constructions define a double pseudofunctor
  \begin{equation*}
    \Ho\maps \uModel \too \uCat
  \end{equation*}
\end{thm}
\begin{proof}
  We take the constraint 2-cells to be those of the pseudofunctors
  \bL\ and \bR\ defined in \S\ref{sec:der-func}.  The ordinary
  pseudofunctor coherence axioms
  (\autoref{defn:dbl-psfr}\ref{item:psfr-ax}) follow from naturality
  of fibrant and cofibrant replacement, just as for the ordinary
  pseudofunctors \bL\ and \bR.

  Proving the double-naturality axioms is basically an exercise in
  filling up big diagrams with lots of naturality squares, though we
  do have to take care that enough ``interior'' arrows are weak
  equivalences that the diagram can be chased in the homotopy
  category.  The diagrams for~\eqref{eq:dblnat-1}
  and~\eqref{eq:dblnat-2} are shown in Figures~\ref{fig:dblnat-1}
  and~\ref{fig:dblnat-2}, respectively.  In both cases the source and
  target of the zigzags in question are placed in boxes to be easily
  visible, and all the quadrilaterals are naturality squares.  In
  Figure~\ref{fig:dblnat-1} the two zigzags go around the top-right
  and the bottom-left, and the marked arrows are weak equivalences.
  In Figure~\ref{fig:dblnat-2} the two zigzags go across the top and
  the bottom, and \emph{all} the arrows are weak equivalences.  The
  zigzag along the bottom of Figure~\ref{fig:dblnat-2} represents the
  identity since a backwards-pointing arrow represents the inverse of
  its forward-pointing version.
\end{proof}
\suppressfloats[t]
  \begin{figure}[t]
    \[\vcenter{\xymatrix@-.5pc{
        wQgRfR & wQgQRfR \ar[l]_\sim \ar[r] &
        wgQRfR \ar[r]^\beta &
        kvQRfR \ar[r] &
        kRvQRfR \\
        \framebox{$wQgfR$} \ar[u]^\sim &
        wQgQfR \ar[l]_\sim \ar[r] \ar[u]_\sim &
        wgQfR \ar[u] \ar[r]^\beta &
        kvQfR \ar[u] \ar[r] &
        kRvQfR \ar[u]_\sim \\
        &
        wQgQfQR \ar[dl]^\sim \ar[r] \ar[u]^\sim &
        wgQfQR \ar[u] \ar[d] \ar[r]^\beta &
        kvQfQR \ar[u] \ar[d] \ar[r] &
        kRvQfQR \ar[u]_\sim \ar[d] \\
        wQgfQR \ar[uu]^\sim \ar[rr] &&
        wgfQR \ar[r]^\beta &
        kvfQR \ar[r] \ar[d]_\alpha &
        kRvfQR \ar[d]^\alpha \\
        &&& khuQR \ar[r] \ar[d] &
        kRhuQR \ar[d] \\
        &&& khRuQR \ar[r]^-\sim &
        kRhRuQR\\
        &&& khRuQ \ar[u]^\sim \ar[r]_-\sim&
        \framebox{$kRhRuQ$}\ar[u]_\sim
      }}
    \]
    \caption{Proof of~\eqref{eq:dblnat-1} for the homotopy double pseudofunctor}
    \label{fig:dblnat-1}
  \end{figure}
  \begin{figure}[t]
    \[\vcenter{\xymatrix@C=3pc@R=1pc{
        &
        fQQR \ar[dl]_{fQ\pi R} \ar[r]^{f\pi QR} &
        fQR \ar[r]^{\rho fQR} &
        RfQR \\
        fQR &&&& RfQ \ar[ul]_{RfQ\rho} \\
        \framebox{$fQ$} \ar[u]^{fQ\rho} &&
        fQQ \ar[uul]^(.6){fQQ\rho} \ar[ll]^{fQ\pi} \ar[rr]_{f\pi Q} &&
        \framebox{$fQ$} \ar[uull]^(.6){fQ\rho} \ar[u]_{\rho fQ} \\
        && fQQQ \ar[dl]_{fQ\pi Q} \ar[dr]^{f\pi QQ} \ar[u]_{fQQ\pi} \\
        &fQQ \ar[luu]^{fQ\pi} \ar[dr]_{f\pi Q} &&
        fQQ \ar[uur]_{fQ\pi} \ar[dl]^{f \pi Q} \\
        && fQ
      }}
    \]
    \caption{Proof of~\eqref{eq:dblnat-2} for the homotopy double pseudofunctor}
    \label{fig:dblnat-2}
  \end{figure}

We can now fulfill our promise to exhibit the preservation of
adjunctions as a functoriality statement.

\begin{cor}
  If $f\maps \sC\toot\sD\spam g$ is a Quillen adjunction, then we have
  a derived adjunction $\bL f\maps \Ho(\sC)\toot\Ho(\sD)\spam \bR g$.
\end{cor}
\begin{proof}
  This follows from \autoref{thm:htpy-dbl-psfr} and
  \autoref{thm:conj-dblpsfr}.
\end{proof}

In fact, applying the functor $\cConj$ to the morphism $\Ho$ in \fDbl,
we obtain the ordinary pseudofunctor defined
in~\cite[1.4.3]{hovey:modelcats}.
\[\cConj(\Ho)\maps \cConj(\uModel)\too\cConj(\uCat).
\]

Considering companion pairs instead, we obtain the following dual
result, which is also occasionally useful:

\begin{cor}
  If $f\maps \sC\to\sD$ is both left and right Quillen (with respect
  to the same model structures), then $\bL f\iso \bR f$.\qed
\end{cor}

The full power of \autoref{thm:htpy-dbl-psfr}, however, lies in the
fact that double pseudofunctors also preserve mates
(\autoref{thm:mate-dblpsfr}).  In
\S\S\ref{sec:beck-chev-cond}--\ref{sec:projfmla-ex} we will see how to
apply this fact to compare composites of left and right derived
functors.

\begin{rmk}
  Note that the problem mentioned at the end of \S\ref{sec:der-func}
  does not arise in the double-categorical context.  It is perfectly
  possible to have a 2-cell $\alpha$ in \uModel\ which is a natural
  isomorphism, but for which $\Ho(\alpha)$ is \emph{not} an
  isomorphism.  This is because the fact that $\alpha$ is an
  isomorphism in \cCat\ is not visible to the double category
  \uModel, and hence need not be preserved by the functor \Ho.
\end{rmk}

\begin{rmk}
  \autoref{thm:htpy-dbl-psfr} admits various generalizations.  For
  instance, it is shown in~\cite{hovey:modelcats} that if \sV\ is a
  monoidal model category, then the homotopy category of any \sV-model
  category is enriched over $\Ho(\sV)$; in this way we can construct a
  double pseudofunctor from \sV-model categories to
  $\Ho(\sV)$-categories.  We could also lift the codomain of $\Ho$ to
  the double category $\uQCat=\uSq(\cQCat)$ of quasicategories
  (see~\cite{joyal:q_kan,joyal:quasi,lurie:higher-topoi}), where
  \cQCat\ is the 2-category of quasicategories described
  in~\cite{joyal:quasi}.  A third generalization is described in the
  next section.
\end{rmk}

\section{Derivable categories}
\label{sec:derivable-categories}

For many purposes, the powerful framework of model categories and
Quillen adjunctions is adequate, but there are some examples in which
it is too restrictive.  This includes many examples where we want to
compare left and right derived functors.  The problem is roughly that
when dealing with derived functors, we need more flexible notions of
``fibrant'' and ``cofibrant'' objects than are supplied by a model
structure.

In this section we describe an extension of the double pseudofunctor
\Ho\ from \uModel\ to a larger double category \uDrv\ of
\emph{derivable categories}.  This generalization also serves to
clarify the essential properties necessary for the definition of
derived functors and the double pseudofunctor \Ho.

\begin{defn}\label{defn:dercat}
  A \textbf{derivable structure} on a category \sC\ consists of:
  \begin{enumerate}
  \item a class of ``weak equivalences'' satisfying the 2-out-of-3
    property,
  \item full subcategories $\sC_Q$ and $\sC_R$,
  \item a functor $Q\maps \sC\to\sC$, whose image is contained in
    $\sC_Q$, and a natural weak equivalence $\pi\maps Q\rightwe
    \Id_\sC$, and
  \item a functor $R\maps \sC\to\sC$, whose image is contained in
    $\sC_R$, and a natural weak equivalence $\rho\maps \Id_\sC\rightwe
    R$, such that
  \item either $Q(\sC_R)\subset \sC_R$ or $R(\sC_Q)\subset
    \sC_Q$.\label{item:hocat-QR}
  \end{enumerate}
  A category equipped with a derivable structure is called a
  \textbf{derivable category}.
\end{defn}

In any derivable category we write $\sC_{QR} = \sC_Q\cap\sC_R$.  The
importance of condition~\ref{item:hocat-QR} is visible in the
following lemma.

\begin{lem}\label{lem:QR-eqv}
  If \sC\ is a derivable category, then every object is connected by a
  zigzag of weak equivalences to an object in $\sC_{QR}$.
\end{lem}
\begin{proof}
  If $Q(\sC_R)\subset\sC_R$, possible zigzags are
  \[\xymatrix{X \ar[r]^-\rho_-\sim &
    RX \ar@{<->}[r]^-{\pi R}_-\sim & QRX} \quad\text{and}\quad
  \xymatrix{X \ar@{<-}[r]^-\pi_-\sim &
    Q X \ar[r]^-{Q\rho}_-\sim & QRX}.\]
  Similarly, if $R(\sC_Q)\subset \sC_Q$, possible zigzags are
  \[\xymatrix{X \ar[r]^-\rho_-\sim &
    RX \ar@{<->}[r]^-{R \pi}_-\sim & RQX} \quad\text{and}\quad
  \xymatrix{X \ar@{<-}[r]^-\pi_-\sim &
    Q X \ar[r]^-{\rho Q}_-\sim & RQX}.\qedhere\]
\end{proof}

\begin{eg}
  Of course, any model category \sC\ is a derivable category if we
  take $\sC_Q$ and $\sC_R$ to be its full subcategories of cofibrant
  and fibrant objects, respectively.  In this case both disjuncts
  of~\ref{item:hocat-QR} can be satisfied at once.
\end{eg}

\begin{eg}\label{eg:drv-triv}
  If \sC\ is any category equipped with a class of weak equivalences
  satisfying the 2-out-of-3 property, we can make it into a derivable
  category with $\sC_Q=\sC_R=\sC$ and $Q=R=\Id$.
\end{eg}

\begin{eg}
  The product of two derivable categories is derivable, with a
  pointwise structure.  Also, the opposite of any derivable category
  is also derivable (simply switch $\sC_Q$ and $\sC_R$).
\end{eg}

We will see other examples of derivable categories in
\S\S\ref{sec:beck-chev-cond}--\ref{sec:projfmla-ex}.

Any derivable category \sC\ has a homotopy category $\Ho(\sC)$
obtained by formally inverting its weak equivalences (though
$\Ho(\sC)$ may not have small hom-sets without additional assumptions
on \sC).  An equivalent homotopy category is obtained by inverting the weak
equivalences in $\sC_Q$, $\sC_R$, or (by \autoref{lem:QR-eqv})
$\sC_{QR}$.

\begin{defn}
  If \sC\ and \sD\ are derivable categories, a functor $f\maps
  \sC\to\sD$ is \textbf{left derivable} if
  \begin{enumerate}
  \item it preserves weak equivalences in $\sC_Q$ and\label{item:htp-lder-i}
  \item $f(\sC_Q)\subset \sD_Q$.\label{item:htp-lder-ii}
  \end{enumerate}
  The dual notion is \textbf{right derivable}.
\end{defn}

Condition~\ref{item:htp-lder-i} ensures that any left derivable
$f\maps \sC\to\sD$ has a left derived functor $\bL f\maps
\Ho(\sC)\to\Ho(\sD)$, defined to be induced by the composite $f\circ
Q\maps \sC\to\sD$.  Condition~\ref{item:htp-lder-ii} ensures that the
composite of two left derivable functors is again left derivable.

\begin{eg}
  Any left Quillen functor between model categories is left derivable.
  Similarly, any right Quillen functor is right derivable.
\end{eg}

\begin{eg}
  If \sC\ and \sD\ are derivable categories in which $Q=R=\Id$, as in
  \autoref{eg:drv-triv}, then a functor $\sC\to\sD$ is left or right
  derivable just when it preserves all weak equivalences.
\end{eg}

\begin{defn}
  We write \uDrv\ for the double category whose objects are derivable
  categories, whose vertical arrows are left derivable functors, whose
  horizontal arrows are right derivable functors, and whose 2-cells
  are arbitrary natural transformations.
\end{defn}

Since every model category is a derivable category and every Quillen
functor is derivable, we have a forgetful functor $\uModel\to\uDrv$.

\begin{thm}\label{thm:hocat-dblpsfr}
  There is a double pseudofunctor
  \[\Ho\maps \uDrv\to\uCat\]
  sending each object \sC\ to $\Ho(\sC)$, each vertical 1-cell $f$ to
  $\bL f$, each horizontal 1-cell $g$ to $\bR g$, and each 2-cell to a
  derived transformation defined as in \autoref{thm:htpy-dbl-psfr}.
\end{thm}
\begin{proof}
  This is a slight generalization of the proof of
  \autoref{thm:htpy-dbl-psfr}.  Pseudofunctoriality in each direction
  follows exactly as in that case.  However, we are now forced to
  choose between the defining composites~\eqref{eq:dernat-composite}
  and~\eqref{eq:dernat-composite-alt} based on whether
  $Q(\sD_R)\subset \sD_R$ or $R(\sD_Q)\subset \sD_Q$, since we have
  only required one or the other to hold.  Note that in either case,
  \autoref{lem:QR-eqv} enables us to use the simpler version of
  \autoref{thm:dernat-simple}.  Additional diagram chases, which
  differ inconsequentially from those in Figure~\ref{fig:dblnat-1},
  are required to verify the double naturality
  axioms for the composite of two 2-cells where one
  uses~\eqref{eq:dernat-composite} and the
  other~\eqref{eq:dernat-composite-alt}.
\end{proof}

This theorem is applied in the same way as
\autoref{thm:htpy-dbl-psfr}.  For example, we have the following
immediate corollaries.

\begin{cor}
  If $f\maps \sC\to\sD$ is a functor which is both left and right
  derivable (relative to the same derivable structures), then $\bL
  f\iso \bR f$.
\end{cor}
\begin{proof}
  Such functors $f$ are precisely the strong morphisms in \uDrv.
\end{proof}

\begin{cor}\label{thm:drv-adjn}
  If $f\adj g$ is an adjunction between derivable categories in which
  $f$ is left derivable and $g$ is right derivable, then we have an
  adjunction $\bL f \adj \bR g$.
\end{cor}
\begin{proof}
  Such an adjunction is precisely a conjunction in \uDrv.
\end{proof}

We call a conjunction in \uDrv\ a \textbf{derivable adjunction}.

Of central importance for us, of course, is that \emph{mates} are
additionally preserved.  The rest of the paper will focus on example
applications of this fact.

\begin{eg}\label{eg:monoidal-model}
  The extra generality of derivable categories and functors can be
  useful even when simply comparing functors of the same handedness.
  For instance, if \sC\ is a \emph{monoidal model category}, then
  although its tensor product $\ten:\sC\times\sC\to\sC$ does satisfy a
  Quillen condition of sorts, it is not a left Quillen functor and not
  a morphism in \uModel.  However, it does preserve cofibrant objects
  and weak equivalences between cofibrant objects, so it is left
  derivable.

  In this way any such \sC\ becomes a \emph{pseudomonoid} in the
  2-category $\cV(\uDrv)$.  (A pseudomonoid is the 2-categorical
  ``internalization'' of a monoidal category.)  Since pseudomonoids
  are preserved by any product-preserving pseudofunctor, it follows
  immediately that $\Ho(\sC)$ is a monoidal category for any monoidal
  model category \sC\ (and more generally, for any ``monoidal
  derivable category'' \sC).
\end{eg}

\begin{rmk}
  In the terminology of~\cite{dhks:holim}, the subcategories $\sC_Q$
  and $\sC_R$ of a derivable category \sC\ are a left and right
  ``deformation retract'' of \sC, respectively, and our derivable
  functors are a special sort of ``deformable functors.''  The
  difference in viewpoint is that we consider $\sC_Q$ and $\sC_R$ to
  be given structure on the category \sC, whereas~\cite{dhks:holim}
  allows deformation retracts to vary with the functors under
  consideration.

  Of particular note is that if $f\colon \sC\toot\sD\spam g$ is a
  \emph{deformable adjunction} in the sense of~\cite{dhks:holim}, then
  it becomes a derivable adjunction in our sense if we choose $\sC_Q$
  to be a ``left $f$-deformation retract,'' $\sD_R$ to be a ``right
  $g$-deformation retract,'' and $\sC_R = \sC$ and $\sD_Q=\sD$.
  Therefore, the results of~\cite[\S44]{dhks:holim} on derived
  adjunctions of deformable adjunctions follow from our
  \autoref{thm:drv-adjn}, and~\cite[44.3]{dhks:holim} is then a
  special case of the preservation of mates by the double
  pseudofunctor $\Ho$.
\end{rmk}

\part{Applications}
\label{sec:applications}

\section{Base change for parametrized spaces}
\label{sec:beck-chev-cond}

As we saw in \S\ref{sec:mates}, a number of important questions can be
phrased in the form ``is the mate of such-and-such a transformation an
isomorphism or not?''  The fact that the double pseudofunctor \Ho\
preserves mates for conjunctions gives us a structured way to attack
such questions at the level of homotopy categories, by giving an
explicit formula for the mate of a derived transformation
$\Ho(\alpha)$---namely, it is the derived transformation of the mate
of $\alpha$.
In the remainder of the paper we present several worked examples of
how to apply this technique, taken both from folklore and from recent
work such as~\cite{maysig:pht,shulman:locconst}.

We begin with a simpler version of the situation
of~\cite{maysig:pht}, where we deal with \emph{unsectioned} spaces.
Let \bTop\ denote the category of compactly generated topological
spaces.  Then for any space $B$ we have a category $\bTop/B$, and for
any continuous $f\maps A\to B$ we have an adjunction
\[f_!\maps \bTop/A \toot \bTop/B\spam f^*,\]
where $f_!$ is given by composition with $f$, and $f^*$ is given by
pullback along $f$.  The categories $\bTop/B$ and functors $f^*$
assemble into a pseudofunctor $\bTop\op \to \cCat$.

\begin{lem}\label{thm:bc-top-pslevel}
  This pseudofunctor satisfies the Beck-Chevalley condition for any
  pullback square
  \[\vcenter{\xymatrix{A \ar[r]^h\ar[d]_f \pullbackcorner & B \ar[d]^g\\
      C\ar[r]_k & D.}}
  \]
  That is, for such a square, the canonical transformation $f_! h^*
  \to k^* g_!$ is an isomorphism.
\end{lem}
\begin{proof}
  This follows from an elementary lemma about pullback squares (and
  thus remains true if \bTop\ is replaced by any category with
  pullbacks).
\end{proof}

Now each category $\bTop/B$ inherits a model structure from the
``classical'' one on \bTop.  The weak equivalences are weak homotopy
equivalences of total spaces, and the fibrations are Serre fibrations
of total spaces (so in particular, the fibrant objects are Serre
fibrations over $B$).  With these model structures, each adjunction
$f_!\adj f^*$ is Quillen, so we have derived adjunctions $\bL f_!\adj
\bR f^*$, and it is natural to ask whether we still have isomorphisms
$\bL f_! \circ \bR h^* \iso \bR k^* \circ \bL g_!$.  This is no longer
true for all pullback squares (see \autoref{thm:not-bc}, below), but
the preservation of mates by the homotopy double pseudofunctor enables
us to give a sufficient condition for it to hold.

\begin{thm}\label{thm:bc-top}
  The derived pseudofunctor $B\mapsto \Ho(\bTop/B)$ satisfies the
  Beck-Chevalley condition for a pullback square
  \[\vcenter{\xymatrix{A \ar[r]^h\ar[d]_f \pullbackcorner & B \ar[d]^g\\
      C\ar[r]_k & D.}}
  \]
  as long as either $g$ or $k$ is a (Serre) fibration.
\end{thm}
\begin{proof}
  We have to show that the mate
  \begin{equation}
    \bL f_! \circ \bR h^* \too \bR k^* \circ \bL g_!\label{eq:bc-top-mate}
  \end{equation}
  of the isomorphism
  \begin{equation}
    \bR h^* \circ \bR g^* \iso \bR f^* \circ \bR k^*\label{eq:bc-top-orig}
  \end{equation}
  is itself an isomorphism.  Because the double pseudofunctor $\Ho$
  preserves mates, and~\eqref{eq:bc-top-orig} is the derived
  transformation of the isomorphism $h^*g^*\iso f^*k^*$, it follows
  that~\eqref{eq:bc-top-mate} is the derived transformation of the
  mate $f_! h^* \to k^* g_!$ (which is an isomorphism by
  \autoref{thm:bc-top-pslevel}).  Therefore, by
  \autoref{thm:dernat-simple},~\eqref{eq:bc-top-mate} is represented
  by the composite
  \[f_!Q h^* X \too f_!h^* X \too[\iso] k^*g_! X \too k^*R g_! X\]
  where $X$ is fibrant and cofibrant in $\bTop/C$, i.e.\ $X$ is a
  cofibrant space and $X\to C$ is a fibration.  We want to show that
  this composite is a weak equivalence.  But $f_!$ preserves all weak
  equivalences, since it is just given by composition, so $f_!Q h^* X
  \too f_!h^* X$ is always a weak equivalence.  Thus, it suffices to
  show that $k^*g_! X \too k^*R g_! X$ is also a weak equivalence;
  here is where we will use the hypothesis on $g$ or $k$.

  We know that $g_!X\to R g_!X$ is a weak equivalence, so it suffices
  to show that this weak equivalence is preserved by the functor
  $k^*$.  If $k$ is a fibration, then this is clear, since pullback
  along a fibration preserves all weak equivalences (i.e.\ \bTop\ is
  right proper).  On the other hand, if $g$ is a fibration, then
  $g_!X\to D$, being the composite $X\to C\xto{g} D$, is also a
  fibration, and thus $g_!X$ is fibrant in $\bTop/D$.  Therefore, since
  weak equivalences between fibrant objects are preserved by right
  Quillen functors, $k^*$ preserves the weak equivalence $g_!X\to
  Rg_!X$, as desired.
\end{proof}

\begin{rmk}
  The same proof applies with any model category replacing \bTop, as
  long as it is either right proper or we assume that the objects $C$
  and $D$ are fibrant.
\end{rmk}

\begin{rmk}
  Note that in the case when $k$ is a fibration, and hence so is $h$,
  all the functors $f^*$, $g^*$, $h^*$, $k^*$, $k_!$, and $h_!$ lie in
  $\cH(\uDrv)$.  Thus, this case of \autoref{thm:bc-top} could be
  deduced from the ordinary pseudofunctoriality of $\bR\maps
  \cH(\uDrv)\to\cCat$.  However, this is no longer the case when it is
  $g$ that is a fibration.
\end{rmk}

\begin{rmk}\label{thm:not-bc}
  The same techniques can also be used to show that a particular
  square \emph{violates} the Beck-Chevalley condition.  For instance,
  consider the situation of~\cite[Counterexample
  0.0.1]{maysig:pht}, where the pullback square is
  \[\vcenter{\xymatrix{\emptyset \ar[r]^h\ar[d]_f \pullbackcorner & \star \ar[d]^{g=1}\\
      \star\ar[r]_-{k=0} & [0,1].}}
  \]
  In this case the derived Beck-Chevalley transformation is represented
  by the composite
  \[f_!Q h^* X \too f_!h^* X \too[\iso] k^*g_! X \too k^*R g_! X
  \]
  where $X$ is a space fibrant and cofibrant over $\star$, i.e.\ just a
  cofibrant space.  Since $\bTop/\emptyset$ is trivial, $Q h^*X \to
  h^*X$ is an isomorphism $\emptyset\iso\emptyset$, and thus so is
  $f_!Q h^* X \to f_!h^* X$.  However, $k^*g_! X$ is also empty,
  whereas $k^*R g_! X$ is not (unless $X$ is itself empty); thus the
  composite cannot be a weak equivalence.
\end{rmk}

The situation of greater interest in \cite{maysig:pht} is more
complicated: instead of the category $\bTop/B$ of spaces over $B$, we
consider the category $\bEx_B$ of spaces over \emph{and} under $B$.
An object of $\bEx_B$, called an \emph{ex-space} over $B$, is a space
$X$ equipped with a projection $p\maps X\to B$ and a section $s\maps
B\to X$ such that $ps=\id_B$.  Once again for any $f\maps A\to B$ we
have an adjunction \[f_!\maps \bEx_A\toot \bEx_B\spam f^*\] where
$f^*$ is given by pullback, except that now $f_!$ is given by pushout
rather than mere composition.

\begin{rmk}\label{rmk:top-lcc}
  To ensure good behavior of pushouts, in the sectioned case we allow
  $X$ to be merely a $k$-space, but the base spaces must still be
  compactly generated; see~\cite[\S1.3]{maysig:pht}.
\end{rmk}

Each $\bEx_B$ again inherits a model structure from \bTop, although
in~\cite{maysig:pht} a certain ``$qf$-model structure'' is constructed
with better formal behavior.  For the purposes of the Beck-Chevalley
condition, however, it is most convenient to give $\bEx_B$ the
following derivable structure: we take $(\bEx_B)_R$ to consist of
ex-spaces whose projection is a Hurewicz fibration, and $(\bEx_B)_Q$
to consist of ex-spaces whose section is a fiberwise closed Hurewicz
cofibration.  In the terminology of~\cite{maysig:pht}, $(\bEx_B)_R$
consists of \emph{$h$-fibrant} objects, $(\bEx_B)_Q$ of
\emph{well-sectioned} or \emph{$\bar{f}$-cofibrant} objects, and
$(\bEx_B)_{QR}$ of \emph{ex-fibrations}.  In~\cite[\S8.3]{maysig:pht}
it is shown that there are functors $Q$ and $R$ making $\bEx_B$ into a
derivable category in this way.

By~\cite[8.2.2]{maysig:pht}, each functor $f_!$ preserves
well-sectioned ex-spaces, and by \cite[7.3.4]{maysig:pht} it preserves
weak equivalences between well-sectioned ex-spaces; thus it is left
derivable.  On the other hand, each functor $f^*$ certainly preserves
Hurewicz fibrations and weak equivalences between them; hence it is
right derivable.  It follows that each adjunction $f_!\adj f^*$ is a
derivable adjunction (i.e.\ a conjunction in \uDrv).

We now upgrade the proof of the Beck-Chevalley condition for ex-spaces
given in~\cite[9.4.6]{maysig:pht}, making the use of
\autoref{thm:hocat-dblpsfr} explicit and thus showing that the
isomorphism constructed is, in fact, the canonical Beck-Chevalley
transformation.

\begin{thm}
  The derived pseudofunctor $B\mapsto \Ho(\bEx_B)$ satisfies the
  Beck-Chevalley condition for a pullback square
  \[\vcenter{\xymatrix{A \ar[r]^h\ar[d]_f \pullbackcorner & B \ar[d]^g\\
      C\ar[r]_k & D.}}
  \]
  as long as either $g$ or $k$ is a Serre fibration.
\end{thm}
\begin{proof}
  As in the proof of \autoref{thm:bc-top}, we must show that the composite
  \[f_!Q h^* X \too f_!h^* X \too[\iso] k^*g_! X \too k^*R g_! X
  \]
  is a weak equivalence, where $X$ is an ex-fibration.  Now
  by~\cite[8.2.2]{maysig:pht}, $h^*$ preserves ex-fibrations, so in
  particular $h^*X$ is well-sectioned.  Thus $Qh^*X\to h^*X$ is a weak
  equivalence between well-sectioned ex-spaces, and so it is preserved
  by $f_!$.  It remains to show that the weak equivalence $g_! X \to
  R g_! X$ is preserved by $k^*$, and as before, this is evident if
  $k$ is itself a fibration.  If instead $g$ is a fibration, we factor
  $k$ as a homotopy equivalence followed by a Hurewicz fibration and
  consider the two cases separately.  The second case we have already
  dealt with, whereas if $k$ is a homotopy equivalence, then since $g$
  is a Serre fibration, $h$ is also a homotopy equivalence.  Hence
  by \cite[7.3.4]{maysig:pht}, the adjunctions $\bL h_!\adj \bR h^*$
  and $\bL k_!\adj \bR k^*$ are adjoint equivalences, and so in the
  composite
  \begin{equation}
    \xymatrix@C=3pc{\bL f_! \bR h^* \ar[r]^-{\eta \bL f_! \bR h^*} &
      \bR k^* \bL k_! \bL f_! \bR h^* \iso
      \bR k^* \bL g_! \bL h_! \bR h^* \ar[r]^-{\bR k^* \bL g_!\ep} &
      \bR k^* \bL g_!}\label{eq:bc-ex-othercomp}
  \end{equation}
  both $\eta$ and $\ep$ are isomorphisms.  (Note that this composite
  is not the composite we have taken to \emph{define} the
  Beck-Chevalley map; that would be
  \begin{equation}
    \xymatrix@C=3pc{\bL f_! \bR h^* \ar[r]^-{\bL f_! \bR h^*\eta} &
      \bL f_! \bR h^* \bR g^* \bL g_! \iso
      \bL f_! \bR f^* \bR k^* \bL g_! \ar[r]^-{\ep \bR k^* \bL g_!} &
      \bR k^* \bL g_!.}\label{eq:bc-ex-official}
  \end{equation}
  However, they are equal, because the isomorphism $\bL k_! \bL f_!
  \iso \bL g_! \bL h_!$ occurring in~\eqref{eq:bc-ex-othercomp} is the
  mate of the isomorphism $\bR h^* \bR g^*\iso \bR f^* \bR k^*$
  occurring in~\eqref{eq:bc-ex-official}.)
\end{proof}

Analogous proofs apply to the study of the category $\bSp_B$ of
\emph{ex-spectra} over $B$; we leave the rephrasing of these to the
interested reader.  Since the derived versions of $f_!$ and $f_*$ for
ex-spectra are parametrized versions of homology and cohomology, these
compatibility relations imply important calculational results.

\section{Base change for sheaves}
\label{sec:sheaves}

The examples in the previous section concerned \emph{spaces over
  spaces} as one way to to do homotopy theory over a base space.
Another widespread type of homotopy theory over a base space studies
\emph{sheaves} of various sorts.  There are many different types of
sheaves, of course, but almost all of them eventually require
comparisons of left and right derived functors.  For simplicity, we
will consider only the category of sheaves of abelian groups on a
topological space $A$, which we denote $\bSh(A)$.  We leave it to the
reader to apply the same language to sheaves on ringed spaces or
topoi, quasicoherent sheaves, simplicial sheaves, sheaves of spectra,
and so on.

The most noticeable difference between all sheaf-theoretic contexts
and that of spaces over spaces is that for a map $f\maps A\to B$ of
base spaces, the pullback functor $f^*\maps \bSh(B)\to \bSh(A)$ of
sheaves always has a well-behaved \emph{right} adjoint $f_*$, rather
than a left adjoint $f_!$.  Furthermore, the Beck-Chevalley condition
for these adjoints does not hold for all pullback squares even on the
point-set level: given a pullback square
\begin{equation}
  \vcenter{\xymatrix{A \ar[r]^h\ar[d]_f \pullbackcorner & B \ar[d]^g\\
      C\ar[r]_k & D.}}\label{eq:pbct-square}
\end{equation}
the Beck-Chevalley transformation $k^* g_* \to f_* h^*$ is only an
isomorphism under additional hypotheses.  Probably the most well-known
and useful result along these lines is the following, which is a
special case of the \emph{Proper Base Change Theorem}.

\begin{lem}[{\cite[2.5.11]{ks:shvs-mfds}}]\label{thm:pbct-pointset}
  If $g$ (and hence also $f$) is a proper map
  in~\eqref{eq:pbct-square} and all spaces involved are locally
  compact Hausdorff, then the Beck-Chevalley transformation $k^* g_*
  \to f_* h^*$ for sheaves is an isomorphism.
\end{lem}

Since the derived version of $f_*$ gives the sheaf-theoretic approach
to cohomology, it is again of importance when and whether this
isomorphism is preserved by passage to homotopy categories.  This is,
of course, a special case of the derived version of the Proper Base
Change Theorem; a very classical argument can be found (for instance)
in~\cite[2.6.7]{ks:shvs-mfds}.  Just as for spaces over spaces, the
preservation of mates by passage to derived functors is an implicit
ingredient in any proof of this result.  Here we sketch one such
proof, making this dependence explicit.

We write $\bCh^+(A)$ for the category of bounded below cochain
complexes of sheaves of abelian groups on $A$.  This category has a
model structure in which the weak equivalences are the
quasi-isomorphisms (homology isomorphisms), every object is cofibrant,
and the fibrant objects are the complexes of injectives.  Each
continuous map $f\maps A\to B$ induces an adjunction $f^*\maps
\bCh^+(B)\toot \bCh^+(A)\spam f_*$ which is Quillen with respect to
these model structures, so we have a derived adjunction $\bL f^*\adj
\bR f_*$.  (In fact, $f^*$ is exact and hence preserves all weak
equivalences, so one usually writes simply $f^*$ instead of $\bL
f^*$.)

\begin{thm}\label{thm:pbct-derived}
  If $g$ in~\eqref{eq:pbct-square} is proper and all spaces involved
  are locally compact Hausdorff, then the derived Beck-Chevalley
  transformation
  \[\bL k^* \circ \bR g_* \too \bR f_* \circ \bL h^*\]
  is an isomorphism.
\end{thm}
\begin{proof}
  The model structures mentioned above are sufficient for defining the
  derived functors, but for this proof we need to give $\bCh^+(A)$ a
  different derivable structure.  Recall that a sheaf $X$ on a space
  $A$ is \emph{c-soft} if sections of $X$ over compact subsets of $A$
  can be extended to all of $A$, i.e.\ if $\Gamma(A,X) \to
  \Gamma(K,X)$ is surjective for all compact $K\subset A$.  Given a
  continuous map $f\maps A\to B$, a sheaf $X$ on $A$ is said to be
  \emph{$f$-soft} if its restriction to every fiber of $f$ is c-soft.
  We say that a bounded below complex of sheaves is c-soft or $f$-soft
  if it consists of c-soft or $f$-soft sheaves.  Since every injective
  sheaf is c-soft and every c-soft sheaf is $f$-soft, every bounded
  below complex of sheaves has a c-soft or $f$-soft resolution.

  In particular, $\bCh^+(B)$ is a derivable category with $\bCh^+(B)_Q
  = \bCh^+(B)$ and $\bCh^+(B)_R$ the full subcategory of $g$-soft
  complexes, and likewise $\bCh^+(A)$ is a derivable category with
  $\bCh^+(A)_Q = \bCh^+(A)$ and $\bCh^+(A)_R$ the full subcategory of
  $f$-soft complexes.  We consider $\bCh^+(C)$ as a derivable category
  with $\bCh^+(C)_Q = \bCh^+(C)_R= \bCh^+(C)$, and likewise for
  $\bCh^+(D)$.  In all cases, the weak equivalences are the
  quasi-isomorphisms.

  Now $g_*$ preserves weak equivalences between $g$-soft complexes,
  and likewise for $f_*$ and $f$-soft complexes, so $g_*$ and $f_*$
  are right derivable functors.  Since the pullback functors $f^*$,
  $g^*$, $h^*$, $k^*$ preserve all weak equivalences, they are left
  derivable, so the mate correspondence in question takes place in
  \uDrv.  Therefore, the derived Beck-Chevalley transformation is the
  derived natural transformation of the point-set-level transformation
  $k^* g_*\to f_* h^*$, and thus is represented by the explicit
  composite
  \[ k^* Q g_* X \too k^* g_* X \too[\iso] f_* h^* X \too f_* R h^* X
  \]
  where $X\in\bCh^+(C)_{QR}$, i.e.\ $X$ is a $g$-soft complex of
  sheaves on $C$.  However, $Q$ is the identity functor, and the
  point-set transformation $k^* g_*\to f_* h^*$ is an isomorphism by
  \autoref{thm:pbct-pointset}, so it remains to show that $f_* h^* X
  \to f_* R h^* X$ is a weak equivalence.  This follows from the
  observation that $h^*$ takes $g$-soft sheaves to $f$-soft ones,
  which is true since the fibers of $f$ are the same as the fibers of
  $g$, and $h^*$ doesn't change the restrictions of $X$ to fibers.
\end{proof}

\begin{rmk}
  The category $\bCh(A)$ of \emph{unbounded} chain complexes of
  sheaves also admits a model structure in which the weak equivalences
  are the quasi-isomorphisms, every object is cofibrant, and the
  fibrant objects are the ``dg-injective'' or ``K-injective''
  complexes;
  see~\cite{joyal:unbdd-cxs-modelstr,spaltenstein:unbdd,beke:sheafifiable,gillespie:kap-der}.
  (This is true with any Grothendieck abelian category replacing
  $\bSh(A)$.)  However, it seems that \autoref{thm:pbct-derived} fails
  for unbounded complexes; see~\cite[6.5.4.3]{lurie:higher-topoi},
  where it is claimed that the solution is to use a more restrictive
  notion of weak equivalence than quasi-isomorphism.
\end{rmk}

The full version of the Proper Base Change Theorem involves defining a
new ``direct image with proper supports'' functor $f_!$ equipped with
a map $f_!\to f_*$ (which is an isomorphism when $f$ is proper), and
showing that the restricted transformation $\bL k^* \circ \bR g_!
\too \bR f_!  \circ \bL h^*$ is an isomorphism whether or not $f$ is
proper.  (This $f_!$ is not a left adjoint of $f^*$, but it serves a
similar function to the $f_!$ for parametrized spaces and spectra.
Whereas the $f_!$ for parametrized spectra is a version of homology,
this $f_!$ is a version of compactly supported cohomology, which plays
a similar role in duality theory.)  While this transformation is no
longer itself defined directly as a mate, mate correspondences are
still essential for comparing various different transformations
relating these functors; see~\cite{fhm:left-and-right}.

\section{Comparing homotopy theories}
\label{sec:comp-homot-theor}

One of the most important uses of model category theory is that it
provides a structured way to compare different ``models'' of the same
or similar homotopy theories.  The strongest such comparison is, of
course, a Quillen equivalence.  For example, there are many different
model categories of spectra, all of which are connected by a web of
Quillen equivalences (see~\cite{mmss,mm:eqosp-smod}); thus one can
work with whichever category of spectra is most appropriate for a
particular application.  However, in situations such as those
considered in the previous two sections, we often have not just two
homotopy theories, but two ``families'' of homotopy theories, each
consisting of many model categories related by base change functors.
Thus, we naturally want to compare not only the model categories
themselves in the two theories, but the base change functors as well.

For instance, one might hope that each ordinary model category of
spectra (or at least some of them) would have a parametrized version
$\bSp_B$ over any base space, with base change adjunctions $f_!\adj
f^*\adj f_*$ induced by any continuous map $f\maps A\to B$.  Given
some other family of categories of parametrized spectra---denoted
$\bSp_B'$, say---one would like to know not only that there are
Quillen equivalences $\bSp_B \toot \bSp_B'$, but that these
equivalences are compatible with the base change adjunctions, and thus
in particular compute the same notions of homology and cohomology.

In fact, only parametrized orthogonal spectra are studied at any length
in~\cite{maysig:pht}, and there are certain difficulties involved in
extending other models of spectra to the parametrized context
(see~\cite[Chapter 24]{maysig:pht}).  To get across the ideas involved
in such comparisons, therefore, we will consider a less hypothetical,
and also much less complicated, example of the same phenomenon:
comparing different model structures for parametrized spaces.

In \S\ref{sec:beck-chev-cond} we considered the model structure on
$\bTop/B$ induced from the ``Quillen'' or $q$-model structure on
\bTop, which is constructed from weak homotopy equivalences, Serre
fibrations, and relative cell complexes.  However, there are also
other model structures on \bTop\ which induce model structures on
$\bTop/B$.  In~\cite{strom:the-htpy-cat} it is shown that there is a
``Hurewicz'' or $h$-model structure constructed from homotopy
equivalences, Hurewicz fibrations, and Hurewicz cofibrations (see
also~\cite{schwanzl-vogt:strong-cofib}, \cite[Ch.~4]{maysig:pht}, and
\cite{cole:many-htpy-cats}).  And in~\cite{cole:mixed} it is shown
that there is also a ``mixed'' or $m$-model structure constructed from
weak homotopy equivalences, Hurewicz fibrations, and maps of the
homotopy type of relative cell complexes.  We have a pair of Quillen
adjunctions
\[\xymatrix{\bTop^q
  \ar@<1mm>[r]^{I^m_q} \ar@{<-}@<-1mm>[r]_{J^m_q} &
  \bTop^m
  \ar@<1mm>[r]^{I^h_m} \ar@{<-}@<-1mm>[r]_{J^h_m} &
  \bTop^h,}\]
where all the functors are the identity functor.
Each of these model structures lifts to $\bTop/B$, and we have
analogous Quillen adjunctions:
\begin{equation}
  \xymatrix{(\bTop/B)^q
    \ar@<1mm>[r]^{I^m_q} \ar@{<-}@<-1mm>[r]_{J^m_q} &
    (\bTop/B)^m
    \ar@<1mm>[r]^{I^h_m} \ar@{<-}@<-1mm>[r]_{J^h_m} &
    (\bTop/B)^h.}\label{eq:modelstr-top-qadj}
\end{equation}
In both cases the first adjunction, but not the second, is a Quillen
equivalence.
(The second is a \emph{left Quillen embedding} in the sense
of~\cite{shulman:locconst}, i.e.\ the derived left adjoint
$\Ho(\bTop^m) \to \Ho(\bTop^h)$ is full and faithful.)

Now the adjunctions $f_!\maps \bTop/A\toot \bTop/B\spam f^*$ are
Quillen relative to all three model structures, and the important
questions regard the commutativity of squares such as
\begin{equation*}
\vcenter{\xymatrix@-.5pc{(\bTop/B)^m \ar[r]^{f^*}\ar[d]_{J^m_q} \ar@{}[dr]|? &
    (\bTop/A)^m \ar[d]^{J^m_q}\\
    (\bTop/B)^q\ar[r]_{f^*} &
    (\bTop/A)^q}}
\quad
\text{and}
\quad
\vcenter{\xymatrix@-.5pc{(\bTop/A)^m \ar[r]^{f_!}\ar[d]_{J^m_q} \ar@{}[dr]|? &
    (\bTop/B)^m \ar[d]^{J^m_q}\\
    (\bTop/A)^q\ar[r]_{f_!} &
    (\bTop/B)^q,}}
\end{equation*}
particularly after passage to homotopy categories.  In general, the
point-set level question may already be interesting, but in our toy
example it is trivial: since all the functors
in~\eqref{eq:modelstr-top-qadj} are the identity functor, all such
squares obviously commute on the point-set level.

Moreover, the commutativity of some of these squares at the homotopy
category level is also easy: since $f^*$, $J^m_q$, and $J^h_m$ are
right Quillen, ordinary pseudofunctoriality of $\bR\maps \cModelr \to
\cCat$ gives us isomorphisms
\begin{align}
  \bR^q f^* \circ \bR J^m_q &\iso \bR J^m_q \circ \bR^m f^*
  \qquad\text{and}\\
  \bR^m f^* \circ \bR J^h_m &\iso \bR J^h_m \circ \bR^h f^*
\end{align}
(where we decorate $\bR$ to indicate which model structure we are
taking the derived functor with respect to).  Likewise, the left
derived functors of $f_!$ commute with those of $I^m_q$ and $I^h_m$:
\begin{align}
  \bL^m f_! \circ \bL I^m_q &\iso \bL I^m_q \circ \bL^q f_!
  \qquad\text{and}\\
  \bL^h f_! \circ \bL I^h_m &\iso \bL I^h_m \circ \bL^m f_!.
\end{align}
Note that these isomorphisms are actually mates of the previous ones.
We also have four additional mates relating composites of left and
right derived functors:
\begin{align}
  \bL^q f_! \circ \bR J^m_q &\too \bR J^m_q \circ \bL^m f_! \label{eq:shrk-J-mq}\\
  \bL^m f_! \circ \bR J^h_m &\too \bR J^h_m \circ \bL^h f_! \label{eq:shrk-J-hm}\\
  \bL I^m_q \circ \bR^q f^* &\too \bR^m f^* \circ \bL I^m_q \label{eq:star-I-mq} \qquad\text{and}\\
  \bL I^h_m \circ \bR^m f^* &\too \bR^h f^* \circ \bL I^h_m \label{eq:star-I-hm}.
\end{align}
(Each of these can actually be constructed in two different ways,
which give the same result by \autoref{lem:adj-composite-mate}.)  Now
since $I^m_q \adj J^m_q$ is a Quillen \emph{equivalence}, its derived
adjunction $\bL I^m_q \adj \bR J^m_q$ is an adjoint equivalence, so
\autoref{thm:adjeqv-mate-iso} implies that~\eqref{eq:shrk-J-mq}
and~\eqref{eq:star-I-mq} are also isomorphisms.  Thus, it remains to
consider~\eqref{eq:shrk-J-hm} and~\eqref{eq:star-I-hm}.  (Of course,
the $h$- and $q$-model structures can be compared by simply composing
the other two comparisons; double pseudofunctoriality shows that this
gives the same result as a direct treatment using $I^h_q = I^h_m \circ
I^m_q$ and $J^h_q = J^m_q\circ J^h_m$.)

\begin{rmk}
  The situation with the $q$ and $m$ model structures is quite common,
  since the most desirable (and very frequently occurring) comparison
  between homotopy theories is a Quillen equivalence.  Thus,
  \autoref{thm:adjeqv-mate-iso} means that in most cases no fancy
  technology is required to compare functors.  However, comparisons
  which are not equivalences do occur.  The $h$ versus $m$/$q$ model
  structures is a fairly trivial example, but we will mention a more
  contentful one below.
\end{rmk}

\begin{prop}
  The derived transformation~\eqref{eq:shrk-J-hm} is an isomorphism.
\end{prop}
\begin{proof}
  By double pseudofunctoriality,~\eqref{eq:shrk-J-hm} is represented
  by the composite
  \[f_! (Q^m X) \too f_! X \too R^h( f_! X)
  \]
  where $X$ is $h$-fibrant and $h$-cofibrant over $A$, i.e.\ $X$ is an
  $h$-cofibrant space (a vacuous condition) and $X\to A$ is a Hurewicz
  fibration.  Since $f_!$ preserves all weak equivalences of any sort,
  the first map in this composite is a weak homotopy equivalence
  (i.e.\ a weak equivalence in the mixed model structure).  And the
  second map is by definition a homotopy equivalence, and therefore
  also a weak homotopy equivalence.
\end{proof}

On the other hand,~\eqref{eq:star-I-hm}
is \emph{not}, in general, an isomorphism.  Double pseudofunctoriality
tells us that it can be represented by the composite
\[Q^m(f^*X) \too f^*X \too f^* (R^h X)
\]
where $X$ is $m$-fibrant and $m$-cofibrant over $B$, i.e.\ $X$ is
$m$-cofibrant (of the homotopy type of a CW complex) and $X\to B$ is
an $m$-fibration.  Since $m$-fibrations are the same as Hurewicz
fibrations, $X$ is already $h$-fibrant; thus $X\to R^h X$ is an
$h$-equivalence between $h$-fibrant objects and so is preserved by
$f^*$.  However, if $f^*X$ is not $m$-cofibrant, then the \emph{weak}
homotopy equivalence $Q^m(f^*X) \to f^*X$ will not be an
$h$-equivalence (otherwise, $f^*X$ would be homotopy equivalent to the
$m$-cofibrant $Q^m(f^*X)$, hence $m$-cofibrant itself).  For example,
we could take $B=X=\star$ to be the one-point space and $A=f^*X$ to be
any non-$m$-cofibrant space.  Thus, in such a case the above composite
is not an $h$-equivalence, and so does not represent an isomorphism in
$\Ho((\bTop/A)^h)$.

\begin{rmk}
  The situation for ex-spaces is basically the same; we have the same
  set of Quillen adjunctions, and the same formal arguments apply to
  six out of the eight possible comparisons.  The analogue
  of~\eqref{eq:shrk-J-hm} is also an isomorphism, although the proof
  requires an invocation of the Gluing Lemma, while the analogue
  of~\eqref{eq:star-I-hm} is not an isomorphism.  We leave the details
  to the reader.
\end{rmk}




There are also interesting model structures on $\bTop/B$ that are not
inherited from model structures on \bTop.  One such is the
``$qf$-model structure'' of~\cite[Ch.\ 6]{maysig:pht}.  However, since
the identity functor is a Quillen equivalence between the $qf$-model
structure and the $q$-model structure, all compatibility relations
follow directly as before.

A more interesting question concerns the model structure on $\bTop/B$
constructed in~\cite{ij:ex-spaces}, which in~\cite{shulman:locconst}
we called the $ij$-model structure.  In this model structure the
cofibrations are built out of cells of the form $U\times D^n$, where
$U$ is any open subset of $B$.  Thus, fibrations and weak equivalences
are detected by ``spaces of sections'' over open sets; precise
definitions can be found in~\cite{ij:ex-spaces}.  Since this model
structure has a very sheaf-theoretic feel, it is unsurprising that it
is Quillen equivalent to a model category of simplicial presheaves on
$B$.  Furthermore, the pullback functor $f^*\maps \bTop/B\to\bTop/A$
is \emph{left} Quillen for the $ij$-model structures, just as it is
for sheaves, and because we have a Quillen equivalence, the same
arguments as before show that the derived adjunctions $\bL f^*\adj \bR
f_*$ agree for the $ij$-model structure and for simplicial presheaves.

The relationship of the $ij$-model structure to the other model
structures on $\bTop/B$ is more subtle.  We showed
in~\cite{shulman:locconst} that if $B$ is a locally compact CW
complex, then the identity adjunction of $\bTop/B$ is a Quillen
adjunction
\[\iota^\star\maps (\bTop/B)^{ij} \toot (\bTop/B)^m\spam \iota_\star,\]
but not a Quillen equivalence (though it is a \emph{right Quillen
  embedding}, i.e.\ the derived right adjoint is full and faithful).
We showed moreover that for any map $f\maps A\to B$ between locally
compact CW complexes, when the identity $f^*=f^*$ is considered as a
square
\[\vcenter{\xymatrix{(\bTop/B)^m \ar[r]^{\iota_\star}\ar@{=}[d]
    \drtwocell\omit& (\bTop/B)^{ij} \ar[d]^{f^*}\\
    (\bTop/B)^m \ar[r]_{\iota_\star \circ f^*} & (\bTop/A)^{ij}}}
\]
in \uModel, its derived natural transformation is an isomorphism
\[\bL^{ij} f^* \circ \bR\iota_\star \iso \bR \iota_\star \circ \bR^m f^*.\]
The proof involves a careful analysis of the explicit formula for a
derived natural transformation.  We also gave an analysis of the mate
of this isomorphism:
\[\bL\iota^\star\circ \bL^{ij} f^* \too \bR^m f^*\circ \bL\iota^\star\]
which turns out to be an isomorphism when $f$ is a $q$-fibration, but
not much more generally.  By \autoref{thm:hkid-mate-iso}, if $\bR^m
f^*$ happens to have a right adjoint $\bM f_*$ (in~\cite{maysig:pht}
such a right adjoint is constructed for connected ex-spaces or spectra
using Brown representability), the induced transformation
\[\bR\iota_\star \circ  \bM f_* \too \bR^{ij} f_* \circ \bR \iota_\star\]
is also an isomorphism whenever $f$ is a $q$-fibration.  This implies
that parametrized generalized cohomology can be computed via passage to
the $ij$-model structure, one of the main points
of~\cite{shulman:locconst}.  (In fact, this application was the
original motivation for the theory of the present paper.)

\section{The projection formula for sheaves}
\label{sec:mono-funct-proj}

We now consider monoidal structures and monoidal functors, as in
\autoref{eg:closed-monfunc}, beginning this time with sheaves.  The
usual ``injective'' model structure on chain complexes of sheaves is
not a monoidal model structure, and there is no analogue for sheaves
of the ``projective'' model structure for modules over a ring (which
is monoidal), but the remedy for this is well-known: we consider
\emph{flat} resolutions rather than projective ones.

It is proven in~\cite{gillespie:flat-shvs,gillespie:kap-der} that the
category $\bCh(A)$ of unbounded chain complexes of sheaves admits a
\emph{flat model structure} which is monoidal; its cofibrant objects
are \emph{dg-flat} complexes.  (For our purposes, all we actually need
is that all complexes admit dg-flat resolutions, which was proven
by~\cite{spaltenstein:unbdd}.)  Thus $\Ho(\bCh(A))$ is symmetric
monoidal; its tensor product $X\ten^\bL Y$ is represented by $QX\ten
QY$, where $Q$ denotes a dg-flat resolution.  In fact, tensoring with
a dg-flat complex preserves all weak equivalences, so $X\ten^\bL Y$
can equally be represented by $QX\ten Y$ or $X\ten QY$.

Now let $f\maps A\to B$ be a continuous map, which as always induces
an adjunction $f^*\maps \bCh(B)\toot\bCh(A) \spam f_*$.  This
adjunction is derivable for the flat model structures, and moreover
$f^*$ is strong monoidal.  Thus, by pseudofunctoriality in $\cV(\uDrv)$,
its left derived functor $\bL f^*$ is again strong monoidal.

We now regard the isomorphism $f^*(X\ten Y) \iso f^*X \ten f^*Y$, for
fixed $X$, as a transformation
\begin{equation}
\vcenter{\xymatrix{\bCh(B) \ar[r]^{f^*}\ar[d]_{X\ten -}
    \drtwocell\omit & \bCh(A) \ar[d]^{f^*X\ten -}\\
  \bCh(B)\ar[r]_{f^*} & \bCh(A).}}\label{eq:projfmla-orig}
\end{equation}
Thus, under the adjunction $f^*\adj f_*$, it has a mate
\begin{equation}
\vcenter{\xymatrix{\bCh(A) \ar[r]^{f_*}\ar[d]_{f^*X\ten -}
    \drtwocell\omit & \bCh(B)\ar[d]^{X\ten -}\\
    \bCh(A) \ar[r]_{f_*} & \bCh(B).}}\label{eq:projfmla-mate}
\end{equation}
The following is part of the (point-set level) \emph{projection
  formula}, and can be found in~\cite[2.5.13]{ks:shvs-mfds}.

\begin{lem}\label{thm:projfmla-pointset}
  If $f\maps A\to B$ is a proper map of locally compact Hausdorff
  spaces, then the component of~\eqref{eq:projfmla-mate} at any flat
  sheaf $X$ is an isomorphism
  \[X\ten f_*Y\iso f_*(f^*X \ten Y).\]
\end{lem}

Of course, generally of more interest is the derived projection
formula.  We follow our usual procedure of upgrading the standard
proof (in this case, taken from~\cite[2.6.6]{ks:shvs-mfds}) to use the
derived mate correspondence and show that the isomorphism constructed
is, in fact, the canonical comparison map.


\begin{thm}\label{thm:projfmla}
  If $f\maps A\to B$ is a proper map of locally compact Hausdorff
  spaces, then for any bounded below complexes $X$ and $Y$ of sheaves
  on $A$ and $B$, respectively, the isomorphism
  \begin{equation}
    \bL f^*(X\ten^\bL Y) \iso \bL f^*X \ten^\bL \bL f^*Y\label{eq:pf-strmon-der}
  \end{equation}
  (which exhibits $\bL f^*$ as strong monoidal) has a mate
  \begin{equation}
    X \ten^\bL \bR f_* Y \too \bR f_*(\bL f^*X \ten^\bL Y)\label{eq:pf-map}
  \end{equation}
  which is an isomorphism.
\end{thm}
\begin{proof}
  We may assume without loss of generality that $X$ is dg-flat, and
  hence so is $f^*X$ (which is isomorphic to $\bL f^*X$, since $f^*$
  is exact).  Since tensoring with a dg-flat complex preserves all
  weak equivalences, $(X\ten -)$ and $(f^*X\ten -)$ are left derivable
  when $\bCh(A)$ and $\bCh(B)$ are equipped with their
  \emph{injective} model structures (in which everything is cofibrant
  and the fibrant objects are the dg-injective complexes).  With these
  model structures,~\eqref{eq:projfmla-mate} becomes a 2-cell in
  \uDrv\ as drawn, while~\eqref{eq:projfmla-orig} is a 2-cell in
  \uDrv\ where all the arrows involved are vertical:
  \[\vcenter{\xymatrix{
      \bCh(B)\ar@{=}[r]\ar[d]_{X\ten -}  \ddrtwocell\omit &
      \bCh(B) \ar[d]^{f^*}\\
      \bCh(B)\ar[d]_{f^*} &
      \bCh(A)\ar[d]^{f^*X\ten -}\\
      \bCh(A)\ar@{=}[r] &\bCh(A).
    }}\]
  Thus, since~\eqref{eq:pf-strmon-der} is the derived transformation
  of~\eqref{eq:projfmla-orig}, its mate~\eqref{eq:pf-map} is in fact
  the derived natural transformation of~\eqref{eq:projfmla-mate}, and
  is therefore represented by the composite
  \[X\ten Q f_* Y \too X\ten f_* Y \too[\iso]
  f_*(f^*X \ten Y) \too f_* R(f^*X\ten Y)\]
  where $Y$ is dg-injective.  Since we have assumed that $X$ and $Y$
  are bounded below, we may assume that in fact $Y$ is a complex of
  injective sheaves.  Since $(X\ten -)$ preserves all weak
  equivalences, the first questionable map $X\ten Q f_* Y \too X\ten
  f_* Y$ is always a weak equivalence.  Now since $f^*X$ is flat and
  $Y$ is injective (hence c-soft), $(f^*X\ten Y)$ is $f$-soft (this
  follows from~\cite[2.5.12]{ks:shvs-mfds}).  But as we saw in
  \S\ref{sec:sheaves}, $f_*$ preserves weak equivalences between
  $f$-soft complexes, so the second questionable map $f_*(f^*X \ten Y)
  \too f_* R(f^*X\ten Y)$ is also a weak equivalence.
\end{proof}

As with the proper base change theorem, the full version of the
projection formula applies to all maps $f$, but replaces $f_*$ by the
``direct image with proper supports'' functor $f_!$.

\section{The projection formula for parametrized spaces}
\label{sec:projfmla-ex}

Finally, we consider the version of the projection formula proven
in~\cite{maysig:pht} for parametrized spaces.  In this case there is
no known \emph{monoidal} model structure on $\bTop/B$ that is
equivalent to the $q$-model structure, but we can still make do with a
derivable tensor product functor.  Recall that the cartesian monoidal
structure on $\bTop/B$ is given by pullback: $X\ten Y = X\times_B Y$.
Since pullbacks along fibrations preserve weak equivalences, the model
category $(\bTop/B)^q$ is a pseudomonoid in $\cH(\uDrv)$.  (Recall
from \autoref{eg:monoidal-model} that a monoidal model category is, in
particular, a pseudomonoid in $\cV(\uDrv)$.)  Therefore,
$\Ho(\bTop/B)$ (we drop the superscript $q$ from now on) is a monoidal
category; its monoidal structure is represented by $X\times_B^\bR Y =
RX \times_B RY$, where $R$ denotes replacement by a ($q$-)fibration.
Moreover, for any map $f\maps A\to B$, the functor $f^*$ is strong
monoidal and right derivable; thus $\bR f^*$ is again strong monoidal
by ordinary pseudofunctoriality.

Now the isomorphism $f^*X\times_A f^*Y \iso f^*(X\times_B Y)$ making $f^*$
strong monoidal can be viewed as a transformation
\begin{equation}
\vcenter{\xymatrix@C=3pc{
    \bTop/B \ar[d]_{f^*}\ar[r]^{X\times_B -} \drtwocell\omit{^} &
    \bTop/B\ar[d]^{f^*} \\
    \bTop/A\ar[r]_{f^*X\times_A -} &
    \bTop/A}}\label{eq:pf-ex-trans}
\end{equation}
which therefore has a mate under the adjunction $f_!\adj f^*$:
\begin{equation}
\vcenter{\xymatrix@C=3pc{
    \bTop/A\ar[r]^{f^*X\times_A -} \ar[d]_{f_!} \drtwocell\omit &
    \bTop/A\ar[d]^{f_!} \\
    \bTop/B \ar[r]_{X\times_B -} &
    \bTop/B }}\label{eq:pf-ex-mate}
\end{equation}
whose components are transformations
\begin{equation}
  f_!(f^*X \times_A Y) \too X\times_B f_!Y.\label{eq:ex-pf-ps}
\end{equation}

\begin{lem}
  The transformation~\eqref{eq:ex-pf-ps} is always an isomorphism.
\end{lem}
\begin{proof}
  This follows again from elementary facts about pullback squares (and
  hence is true in any category with pullbacks).
\end{proof}

Once again, we are interested in the derived version.

\begin{thm}
  The isomorphism
  \begin{equation}
    \bR f^*X\times^\bR_A \bR f^*Y \iso \bR f^*(X\times^\bR_B Y)\label{eq:pf-ex-dertr}
  \end{equation}
  (which exhibits $\bR f^*$ as strong monoidal) has a mate
  \begin{equation}
    \bL f_!(\bR f^*X \times^\bR_A Y) \too X\times^\bR_B \bL f_!Y\label{eq:pf-ex-dermate}
  \end{equation}
  which is an isomorphism.
\end{thm}
\begin{proof}
  Fix a fibrant $X$, so that $X\times_B-$ and $f^*X\times_A -$ are right
  derivable.  Then~\eqref{eq:pf-ex-mate} is a 2-cell in \uDrv\ as
  drawn, whereas~\eqref{eq:pf-ex-trans} is a 2-cell in \uDrv\ with all
  arrows horizontal:
  \[\vcenter{\xymatrix@C=3pc{
      \bTop/B \ar[r]^{f^*} \ar@{=}[d] \drrtwocell\omit &
      \bTop/A \ar[r]^{f^*X\times_A -} &
      \bTop/ A \ar@{=}[d]\\
      \bTop/B \ar[r]_{X\times_B -} &
      \bTop/B \ar[r]_{f^*} &
      \bTop/A 
    }}
  \]
  Thus, since~\eqref{eq:pf-ex-dertr} is the derived transformation
  of~\eqref{eq:pf-ex-trans}, it follows that its
  mate~\eqref{eq:pf-ex-dermate} is the derived transformation
  of~\eqref{eq:pf-ex-mate}, and thus can be represented by the
  composite
  \[f_! Q(f^*X\times_A Y) \too f_!(f^*X \times_A Y) \too[\iso]
  X\times_B f_!Y \too X\times_B R (f_!Y)
  \]
  where $Y$ is fibrant and cofibrant over $A$.  Now as in
  \S\ref{sec:beck-chev-cond}, $f_!$ preserves all weak equivalences,
  so the first map $f_! Q(f^*X\times_A Y) \to f_!(f^*X \times_A Y)$ is
  a weak equivalence.  Likewise, since $X$ is fibrant, $X\times_B -$
  actually preserves all weak equivalences (that is, $\bTop$ is right
  proper); thus the second map $X\times_B f_!Y \to X\times_B R (f_!Y)$
  is also a weak equivalence.
\end{proof}

The projection formulas for sectioned spaces is more subtle, since in
this case the monoidal structure is itself a composite of left and
right derived functors.  This is the only one of our examples in which
double pseudofunctoriality does not solve all our problems, since we
end up having to compose transformations in a way that is not possible
in the double category \uDrv.  Nevertheless, double
pseudofunctoriality does simplify the problem from a composite of four
transformations to a composite of two, and the rest of the argument is
completed by a diagram chase.  (This diagram chase is exactly the sort
of argument that would have to be made explicit in all the other
examples, if we didn't have double pseudofunctoriality to invoke.  In
other words, Figures~\ref{fig:dblnat-1} and~\ref{fig:dblnat-2} in the
proof of \autoref{thm:htpy-dbl-psfr} have the effect of chasing all
such diagrams once and for all.)

Before defining the fiberwise smash product, we first define the
\textbf{external smash product} of ex-spaces $X\in \bEx_A$ and $Y\in
\bEx_B$ by the following pushout:
\[\vcenter{\xymatrix{(X\times B) \sqcup_{A\times B} (A\times Y) \ar[r]\ar[d] & X\times Y \ar[d]\\
  A\times B\ar[r] & X\exsm Y}}
\]
This produces an ex-space $X\exsm Y\in \bEx_{A\times B}$, and thereby
defines a functor
\[\exsm\maps \bEx_A \times \bEx_B \too \bEx_{A\times B}.\]
Note that the fiber of $X\exsm Y$ over $(a,b)\in A\times B$ is $X_a
\wedge Y_b$.  The external smash product is coherently associative,
unital, and symmetric, and furthermore
we have isomorphisms
\begin{equation}
  f^*X \exsm g^*Y \iso (f\times g)^*(X\exsm Y)\label{eq:pb-exsm}
\end{equation}
satisfying their own coherence conditions.  In the terminology
of~\cite{shulman:frbi}, this structure makes $\bEx_{(-)}$ into a
\emph{monoidal fibration} over \bTop.

We then define the \textbf{fiberwise smash product} of $X, Y\in
\bEx_A$ to be \[X\sm_A Y = \Delta_A^*(X\exsm Y),\] where $\Delta_A\maps
A\to A\times A$ is the diagonal map.  The
isomorphisms~\eqref{eq:pb-exsm}, together with the coherence
isomorphisms of $\exsm$, can be used to construct coherence
isomorphisms making $\sm_A$ a symmetric monoidal structure on $\bEx_A$.
Furthermore, we have isomorphisms
\begin{equation}
  \begin{aligned}
    f^*(X\sm_B Y) &= f^*\Delta_B^*(X\exsm Y)\\
    &\iso \Delta_A^* (f\times f)^*(X\exsm Y)\\
    &\iso \Delta_A^* (f^*X \exsm f^*Y)\\
    &= f^*X \sm_A f^*Y
  \end{aligned}\label{eq:pb-sm-ps}
\end{equation}
making each functor $f^*\maps \bEx_B\to\bEx_A$ strong symmetric
monoidal.

\begin{rmk}
  In~\cite{maysig:pht} this construction is performed in the other
  direction, starting from the fiberwise smash product and producing
  the external one.  See \cite{shulman:frbi} for a full equivalence of
  the two notions.
\end{rmk}

We will use the same derivable structure for ex-spaces as in
\S\ref{sec:beck-chev-cond}, where $(\bEx_A)_{QR}$ consists of
ex-fibrations.  Recall that in this case each adjunction $f_!\adj f^*$
is derivable, and in particular $f^*$ is right derivable.  On the
other hand, it is not hard to see, using the ``gluing lemma,'' that
the functor $\exsm$ is \emph{left} derivable.  Thus, $\sm_A$ is a
composite of a left and a right derivable functor, and so $\bEx_A$ is
not even a pseudomonoid in $\cV(\uDrv)$ or $\cH(\uDrv)$.  We can,
however, produce a symmetric monoidal structure on $\Ho(\bEx_A)$ in
the same way that we did on the point-set level, provided that the
isomorphisms~\eqref{eq:pb-exsm} descend to derived functors.  This is
precisely the argument used in~\cite[\S9.4]{maysig:pht}.

\begin{lem}
  When~\eqref{eq:pb-exsm} is regarded as a 2-cell
  \[\vcenter{\xymatrix{\bEx_C \times \bEx_D
      \ar[r]^{f^*\times g^*}\ar[d]_{\exsm} \drtwocell\omit &
      \bEx_A\times \bEx_B \ar[d]^{\exsm}\\
      \bEx_{C\times D}\ar[r]_{(f\times g)^*} & \bEx_{A\times B}}}
  \]
  in \uDrv, its derived natural transformation is an isomorphism.
\end{lem}
\begin{proof}
  Its derived natural transformation is represented by the composite
  \[Q f^*X\exsm Q g^*Y \too f^*X \exsm g^*Y
  \too[\iso] (f\times g)^*(X\exsm Y)
  \too (f\times g)^*R(X\exsm Y).
  \]
  where $X$ and $Y$ are ex-fibrations over $A$ and $B$, respectively.
  But by~\cite[\S8.2]{maysig:pht}, the functors $f^*$, $g^*$, and
  $\exsm$ all preserve ex-fibrations, so both questionable maps in
  this composite are weak equivalences.
\end{proof}

In particular, we have isomorphisms
\begin{equation}
  \bigl(\bR f^*X\bigr) \exsm^\bL \bigl(\bR g^*Y\bigr)
  \iso \bR(f\times g)^*\left(X\exsm^\bL Y\right).\label{eq:der-pb-exsm}
\end{equation}
The associativity, unit, and symmetry isomorphisms for $\exsm$ descend
to $\exsm^\bL$ (by ordinary pseudofunctoriality), as do the
compatibility axioms between these and~\eqref{eq:pb-exsm} (by double
pseudofunctoriality).  Therefore, we can define the ``middle derived
fiberwise smash product''
\[X \sm_A^\bM Y = \bR \Delta_A^* \left(X \exsm^\bL Y\right)\]
and show:

\begin{thm}
  Each $\sm^\bM$ is a symmetric monoidal structure on $\Ho(\bEx_A)$,
  and each functor $\bR f^*$ is strong symmetric monoidal.
\end{thm}

Now we consider the projection formula.  As in the unsectioned case,
on the point-set level we have a canonical morphism
\begin{equation}
  f_!(f^*X \sm_A Y) \too X\sm_B f_!Y\label{eq:sm-pf-ps}
\end{equation}
defined as the mate of $f^*X\sm_A f^*Y \iso f^*(X\sm_B Y)$, and once
again we have:

\begin{lem}
  The transformation~\eqref{eq:sm-pf-ps} is always an isomorphism.
\end{lem}
\begin{proof}
  This is a straightforward computation with limits and colimits,
  using only the fact that colimits are preserved by pullback in our
  convenient category of spaces (see \autoref{rmk:top-lcc}).
\end{proof}

In order to pass to derived functors, we need to rephrase this in
terms of the external smash product.  The morphism~\eqref{eq:sm-pf-ps}
becomes
\begin{equation}
  f_!\Delta_A^*(f^*X \exsm Y) \too \Delta_B^*(X\exsm f_!Y).\label{eq:exsm-pf-ps}
\end{equation}
The definition of~\eqref{eq:sm-pf-ps} as a mate of the
isomorphism~\eqref{eq:pb-sm-ps} then becomes the following definition
of~\eqref{eq:exsm-pf-ps} in terms of~\eqref{eq:pb-exsm}:
\[\vcenter{\xymatrix{
    \bEx_A \ar@{=}[r]\ar[d]_{f_!} \drtwocell\omit & 
    \bEx_A \ar@{=}[d] &
    \\
    \bEx_B \ar[r]^{f^*}\ar[d]_{X \exsm -} \drtwocell\omit &
    \bEx_A \ar[d]^{f^*X\exsm -} &
    \\
    \bEx_{B\times B} \ar[r]_{(f\times f)^*}\ar@{=}[d] \drrtwocell\omit{\iso} &
    \bEx_{A\times A} \ar[r]_{\Delta_A^*} &
    \bEx_A \ar@{=}[d]\\
    \bEx_{B\times B} \ar[r]_{\Delta_B^*} &
    \bEx_B \ar[r]^{f^*} \ar@{=}[d] \drtwocell\omit &
    \bEx_A \ar[d]^{f_!}\\
    &
    \bEx_B \ar@{=}[r] &
    \bEx_B
    }}
\]
If we choose $X$ to be an ex-fibration (in which case $f^*X$ is also),
then all the 2-cells in this diagram live in \uDrv\ as drawn.
However, this diagram cannot be \emph{composed} in \uDrv, so a single
application of functoriality is insufficient.  Instead, we can break
it into two composites, each of which can be composed in \uDrv.  (In
fact, each is the construction of a mate.)
\begin{equation}
\vcenter{\xymatrix{
    \bEx_A \ar@{=}[r]\ar[d]_{f_!} \drtwocell\omit & 
    \bEx_A \ar@{=}[d]
    \\
    \bEx_B \ar[r]^{f^*}\ar[d]_{X \exsm -} \drtwocell\omit &
    \bEx_A \ar[d]^{f^*X\exsm -}
    \\
    \bEx_{B\times B} \ar[r]_{(f\times f)^*} &
    \bEx_{A\times A} }}
  \quad\text{and}\qquad
  \vcenter{\xymatrix{
    \bEx_{B\times B} \ar[r]^{(f\times f)^*}\ar@{=}[d] \drrtwocell\omit{\iso} &
    \bEx_{A\times A} \ar[r]^{\Delta_A^*} &
    \bEx_A \ar@{=}[d]\\
    \bEx_{B\times B} \ar[r]^-{\Delta_B^*} \ar@{=}[d] \ar@{}[dr]|\iso &
    \bEx_B \ar[r]^{f^*} \ar@{=}[d] \drtwocell\omit &
    \bEx_A \ar[d]^{f_!}\\
    \bEx_{B\times B} \ar[r]_-{\Delta_B}&
    \bEx_B \ar@{=}[r] &
    \bEx_B
    }}\label{eq:expf-broken-2cells}
\end{equation}
Now consider the transformation
\begin{equation}
  \bL f_! \Bigl(\bR \Delta_A^*\Bigl(\bR f^*X \exsm^\bL Y\Bigr)\Bigr) \too
  \bR \Delta_B^*\Bigl(X\exsm^\bL \bL f_!Y\Bigr),\label{eq:der-exsm-pf}
\end{equation}
defined from~\eqref{eq:der-pb-exsm} in the same way
that~\eqref{eq:exsm-pf-ps} is defined from~\eqref{eq:pb-exsm}.  For
the same reasons, it is the composite of the morphisms corresponding
to~\eqref{eq:expf-broken-2cells} at the level of homotopy categories.
But, by pseudofunctoriality, the derived version of each of these is,
in fact, the derived natural transformation of the corresponding
point-set-level transformation (although this doesn't apply to the
entire composite, which can't be composed in \uDrv).
Thus,~\eqref{eq:der-exsm-pf} is equal to the composite
\begin{align*}
  \bL f_! \left(\bR \Delta_A^*\left(\bR f^*X \exsm^\bL Y\right)\right)
  &\too \bL f_! \left(\bR \Delta_A^* \left(\bR \left(f\times f\right)^* \left(X \exsm^\bL \bL f_! Y\right)\right)\right)\\
  &\too\bR \Delta_B^*\left(X\exsm^\bL \bL f_!Y\right).
\end{align*}
Here the first map is $\bL f_! \circ \bR \Delta_A^*$ applied to the
derived transformation of
\[f^*X \exsm Y \too (f\times f)^*(X\exsm f_!Y)\]
and the second is the component at $X\exsm^\bL f_!Y$ of the derived
transformation of
\[f_! \Delta_A^* (f\times f)^* \too \Delta_B^*.
\]
Filling in the definition of derived natural transformations, along
with a pseudofunctoriality constraint for the composite $\Delta_A^*
(f\times f)^*$, we conclude that~\eqref{eq:der-exsm-pf} is represented
by the zigzag along the top-right of the diagram in
Figure~\ref{fig:exproj}, where $Y$ is an ex-fibration.

To show that this composite represents an isomorphism in
$\Ho(\bEx_B)$, we begin by ``following our nose,'' filling in
naturality properties and definitions, arriving at the zigzag along
the bottom-left of Figure~\ref{fig:exproj}.  (The region marked
$\circledast$ is the definition of~\eqref{eq:exsm-pf-ps}; all others
are naturality squares.)

\begin{figure}[!h]
  \resizebox{\textwidth}{!}{$\xymatrix@C=1pc@R=1.5pc{
      \framebox{$f_!Q\Delta_A^*R(f^*X\overline{\wedge}Y)$} \ar[r] &
      f_!Q\Delta_A^*R(f\times f)^*(X \overline{\wedge} f_! Y) \ar[r] &
      f_! Q \Delta_A^* R (f\times f)^* R (X \overline{\wedge} f_! Y) \\
      f_!Q \Delta_A^*(f^*X \overline{\wedge} Y) \ar[u]^\sim \ar[r] \ar[d]_\sim &
      f_!Q\Delta_A^* (f\times f)^*(X\overline{\wedge} f_!Y) \ar[u] \ar[r] \ar[d] &
      f_!Q\Delta_A^*(f\times f)^*R(X\overline{\wedge} f_! Y) \ar[d] \ar[u]_\sim \ar[r]^-\cong &
      f_!Qf^*\Delta_B^*R(X\overline{\wedge} f_! Y) \ar[dd] \\
      f_! \Delta_A^*(f^*X \overline{\wedge} Y)
      \ar[dd]_\cong \ar[r] \ar@{}[ddr]|(.4){\text{\LARGE$\circledast$}} &
      f_!\Delta_A^*(f\times f)^*(X\overline{\wedge} f_!Y) \ar[d]^\cong \ar[r] &
      f_! \Delta_A^*(f\times f)^*R(X\overline{\wedge} f_!Y) \ar[dr]^\cong \\
      & f_!f^*\Delta_B^*(X\overline{\wedge}f_!Y) \ar[dl] \ar[rr] &&
      \Delta_B^* R(X\overline{\wedge} Q f_! Y) \ar[d] \\
      \Delta_B^*(X\overline{\wedge} f_!Y) \ar[rrr]_{} &&&
      \framebox{$\Delta_B^* R(X\overline{\wedge} f_! Y)$}
    }$}
  \caption{The diagram chase for the projection formula}
  \label{fig:exproj}
\end{figure}

Thus it suffices to check that the maps along the bottom-left are weak
equivalences.  For the first two, this is because
$(f^*X\overline{\wedge} -)$ and $\Delta_B^*$ preserve ex-fibrations,
and the third is the isomorphism~\eqref{eq:exsm-pf-ps}.  We deal with
the fourth by replacing $f_! Y$ with an ex-fibration as follows:
\begin{equation}
\vcenter{\xymatrix{
    \Delta_B^*(X\overline{\wedge} f_!Y)
    \ar[r]
    &
    \Delta_B^* R(X\overline{\wedge} f_! Y)
    \\      
    \Delta_B^*(X\overline{\wedge}Qf_!Y)
    \ar[u]^\sim \ar[r] \ar[d]_\sim
    &
    \Delta_B^* R(X\overline{\wedge} Q f_! Y)
    \ar[u]_\sim \ar[d]^\sim
    \\
    \Delta_B^*(X\overline{\wedge} Q R f_!Y)
    \ar[r]_\sim
    &
    \Delta_B^*R(X \overline{\wedge} Q R f_! Y)
}}\label{eq:exproj-addl}
\end{equation}
Here the two left-hand vertical maps are weak equivalences because
when $X$ is an ex-fibration, $(X\wedge_A -) =
\Delta_A^*(X\overline{\wedge} -)$ preserves $h$-equivalences between
well-sectioned ex-spaces (see~\cite[8.2.6]{maysig:pht}).  The two
right-hand vertical maps are weak equivalences because
$(X\overline{\wedge} -)$ is left derivable and $f_!Y$, $Qf_!Y$, and
$QRf_!Y$ are well-sectioned.  Finally, the bottom horizontal map is a
weak equivalence because $QR f_!Y$ is an ex-fibration, hence so is
$X\overline{\wedge} QRf_!Y$.  Thus, we have the projection formula for
ex-spaces:

\begin{thm}
  For any map $f\maps A\to B$, the natural isomorphism
  \begin{equation}
    \bR f^*\Bigl(X \sm_B^\bM Y\Bigr) \iso \Bigl(\bR f^*X \sm_A^\bM \bR f^*Y\Bigr)\label{eq:exproj-strmon}
  \end{equation}
  (which exhibits $\bR f^*$ as strong monoidal) has a mate
  \begin{equation}
    \bL f_! \Bigl(\bR f^* X \sm_A^\bM Y\Bigr) \too \Bigl(X \sm_B^\bM \bL f_! Y\Bigr)\label{eq:exproj}
  \end{equation}
  which is also an isomorphism.\qed
\end{thm}

\begin{rmk}
  The isomorphism given in the proof of the projection formula
  in~\cite[9.4.5]{maysig:pht} is, essentially, the composite of the
  two maps along the left of~\eqref{eq:exproj-addl}:
  \[\Delta_B^*(X\exsm f_! Y) \leftwe \Delta_B^*(X \exsm Q f_! Y)
  \rightwe \Delta_B^*(X \exsm Q R f_! Y).
  \]
  The other weak equivalences in~\eqref{eq:exproj-addl} and
  Figure~\ref{fig:exproj} are implicitly present in the identification
  of the source and target of this zigzag as representing the source
  and target of~\eqref{eq:exproj}.  As always, the contribution of our
  theory is to identify this with the canonical comparison map, i.e.\
  the mate of~\eqref{eq:exproj-strmon}.
\end{rmk}

\begin{rmk}
  In fact, $\bEx_B$ is actually a \emph{closed} monoidal category, and
  it is shown in~\cite[\S9.3]{maysig:pht} (using Brown
  representability) that the subcategory of $\Ho(\bEx_B)$ consisting
  of connected spaces is also closed monoidal.  Now
  \autoref{lem:adj-composite-mate} implies that for a functor between
  closed monoidal categories, the projection formula morphism is the
  same as the map~\eqref{eq:closed-transf-mate} constructed via the
  closed structure.  As remarked in \autoref{eg:closed-monfunc}, it
  then follows from \autoref{thm:hkid-mate-iso} and the projection
  formula that $f^*$ and $\bR f^*$ are \emph{closed} monoidal functors
  (the latter only insofar as $\Ho(\bEx_B)$ is closed).  The analogous
  facts for spectra are true without any connectivity hypothesis.

  These results, which play an important role in~\cite{maysig:pht},
  seem impossible to approach without the technology of mates, since
  the internal-homs in $\Ho(\bEx_B)$ are so inexplicit.  Note in
  particular that for this argument to work, it is essential that the
  isomorphism in the projection formula is not just any isomorphism,
  but the particular map defined as a mate of the derived
  transformation~\eqref{eq:exproj-strmon}.
\end{rmk}

\bibliographystyle{halpha}
\raggedright
\bibliography{all,shulman}

\end{document}